\RequirePackage[l2tabu, orthodox]{nag} 
\documentclass[11pt, a4paper]{article}

\title{Verified eigenvalue and eigenvector computations\\using complex moments and the Rayleigh--Ritz procedure\\for generalized Hermitian eigenvalue problems\thanks{This work is partially supported by the Japan Society for the Promotion of Science grants JP17K12690, JP18H03250, JP18K13453, JP19KK0255, JP20K14356, and JP21H03451.}}
\author{Akira Imakura\thanks{Faculty of Engineering, Information and Systems, University of Tsukuba, 1-1-1 Tennodai, Tsukuba, Ibaraki 305-8573 Japan} \and Keiichi Morikuni\footnotemark[2] \footnotemark[3]\thanks{\texttt{morikuni.keiichi.fw@u.tsukuba.ac.jp}} \and Akitoshi Takayasu\footnotemark[2]}
\date{}

\usepackage{amsmath, amsfonts, amssymb, amsthm}
\usepackage[all, warning]{onlyamsmath} 
\usepackage[bookmarks=true, bookmarksnumbered=true, bookmarkstype=toc, colorlinks={true}, pdfauthor={Keiichi Morikuni}, pdfdisplaydoctitle={true}, pdfkeywords={Add keywords}, pdfpagemode=UseNone, pdfstartview=FitH, pdftitle={Verified eigenvalue and eigenvector computations using complex moments and the Rayleigh-Ritz procedure for generalized Hermitian eigenvalue problems}, setpagesize={false}, urlcolor={red}]{hyperref}
\usepackage[T1]{fontenc}
\usepackage[top=1in, bottom=1in, left=1in, right=1in]{geometry}
\usepackage{exscale, fix-cm, lmodern, textcomp}
\usepackage{algorithm, algorithmic}
\usepackage{empheq}
\usepackage{booktabs}
\usepackage[subrefformat=parens]{subcaption}

\usepackage{mathtools}
\mathtoolsset{showonlyrefs=true}

\hypersetup{pdfpagemode=UseNone} 

\theoremstyle{plain}
\newtheorem{theorem}{Theorem}[section]
\newtheorem{remark}{Remark}[section]
\newtheorem{lemma}{Lemma}[section]

\numberwithin{equation}{section}
\makeatletter
	
	\@addtoreset{equation}{section}
\makeatother

\makeatletter
  
  \@addtoreset{algorithm}{section}
\makeatother

\makeatletter
  
  \@addtoreset{figure}{section}
\makeatother

\makeatletter
  
  \@addtoreset{table}{section}
\makeatother

\allowdisplaybreaks

\begin{document}
\maketitle

\begin{abstract}
We propose a verified computation method for eigenvalues in a region and the corresponding eigenvectors of generalized Hermitian eigenvalue problems.
The proposed method uses complex moments to extract the eigencomponents of interest from a random matrix and uses the Rayleigh--Ritz procedure to project a given eigenvalue problem into a reduced eigenvalue problem.
The complex moment is given by contour integral and approximated using numerical quadrature.
We split the error in the complex moment into the truncation error of the quadrature and rounding errors and evaluate each.
This idea for error evaluation inherits our previous Hankel matrix approach, whereas the proposed method enables verification of eigenvectors and requires half the number of quadrature points for the previous approach to reduce the truncation error to the same order.
Moreover, the Rayleigh--Ritz procedure approach forms a transformation matrix that enables verification of the eigenvectors.
Numerical experiments show that the proposed method is faster than previous methods while maintaining verification performance and works even for nearly singular matrix pencils and in the presence of multiple and nearly multiple eigenvalues.
\end{abstract}



\section{Introduction}
We consider verifying the $t$ eigenvalues $\lambda_i$, counting multiplicity, in a prescribed interval $\Omega = [a, b] \subset \mathbb{R}$ of the generalized Hermitian eigenvalue problem 
\begin{equation}
	A \boldsymbol{x}_i = \lambda_i B \boldsymbol{x}_i, \quad \boldsymbol{x}_i \in \mathbb{C}^n \setminus \{ \boldsymbol{0} \}, \quad \lambda_1 \leq \lambda_2 \leq \cdots \leq \lambda_t,
	\label{eq:gevp}
\end{equation}
where $A = A^\mathsf{H} \in \mathbb{C}^{n \times n}$, $B = B^\mathsf{H} \in \mathbb{C}^{n \times n}$ is positive semidefinite, and the matrix pencil $zB-A$ ($z\in\mathbb{C}$) is regular, i.e, $\det (zB-A)$ is not identically equal to zero for all $z \in \mathbb{C}$; otherwise, it is singular.
We call $\lambda_i$ an eigenvalue and $\boldsymbol{x}_i$ the corresponding eigenvector of the problem \eqref{eq:gevp} or matrix pencil $z B - A$, $z \in \mathbb{C}$ interchangeably and the pair$(\lambda_i, \boldsymbol{x}_i)$ an eigenpair.
Here, the verification of eigenvalues and eigenvectors is to produce rigorous error bounds of numerically computed eigenvalues and eigenvectors, taking into account all possible errors, in particular rounding errors~\cite{Nagatou2003BJSIAM, Rump2010AN}.
Throughout, we assume that the number of eigenvalues in the interval $\Omega$ is known to be $t$ and there do not exist eigenvalues of \eqref{eq:gevp} at the end points $a$, $b \in \mathbb{R}$.
We also denote the eigenvalues of \eqref{eq:gevp} outside $\Omega$ by $\lambda_i$ ($i = t+1, t+2, \dots, r$), where $r = \mathrm{rank}\,B$.
Verified eigenvalue computations arise in applications, e.g., from the numerical verification of a priori error estimations for finite element solutions~\cite{YamamotoNakao1993NumerMath,WatanabeYamamotoNakao1999TJSIAM}, eigenvalues of elliptic operators~\cite{ToyonagaNakaoWatanabe2002JCAM}, and electronic state calculations~\cite{HoshiOgitaOzakiTerao2020JCAM}.

Previous studies of verified eigenvalue and eigenvector computations are classified into two categories: one is for the verification of specific eigenpairs, and the other is for the verification of all the eigenpairs at once.
This study focuses on the former category.

We first review methods in the former category.
Methods in the former have taken several different approaches and typically use fixed-point iterations.
An interval-arithmetic friendly formulation suitable for applying fixed-point iterations can lead to rigorous error bounds.
Yamamoto~\cite{Yamamoto1980NumerMath} and Rump~\cite{Rump1989} regard a given eigenvalue problem as a system of nonlinear equations and use Newton-like iterations for solving the equations~\cite{Krawczyk1969} to verify an eigenpair for nonsymmetric standard and non-Hermitian generalized eigenvalue problems, respectively.
Yamamoto~\cite{Yamamoto1982NumerMath} further introduces a refinement procedure~\cite{SymmWilkinson1980NumerMath}.
Behnke~\cite{Behnke1988,Behnke1991} uses Temple quotients, variational principles, and a generalization of Lehmann's method~\cite{Lekmann1949} for the inclusion of eigenvalues.
An extension of Rump's approach can deal with multiple or nearly multiple eigenvalues and eigenvectors~\cite{Rump2001LAA}.
See \cite[Section~13.4]{Rump2010AN} for a substantial review along this line.
Yamamoto~\cite{Yamamoto2001LAA} uses the $LDL^\mathsf{T}$ and Cholesky decompositions, its error estimation, and Sylvester's law of inertia and develops a method that can also deal with multiple eigenvalues or a cluster of eigenvalues of generalized symmetric eigenvalue problems.

Next, we review methods in the latter category.
Wilkinson~\cite{Wilkinson1961} and Varah~\cite{Varah1968MathComp} use Ger{\v s}hgorin's theorem~\cite{Gershgorin1931} for non-Hermitian matrices.
Oishi~\cite{Oishi2001LAA} uses Bauer--Fike-type and Weyl-type perturbation theorems~\cite{BauerFike1960,HornJohnson2013}.
Maruyama et al.~\cite{MaruyamaOgitaNakayaOishi2004} use Ger{\v s}hgorin's theorem, regards a given eigenvalue problem as a matrix equation, and uses a componentwise error bound~\cite{Yamamoto1984JJAM} and efficient technique~\cite{OishiRump2002NumerMath} for the matrix equation to verify all eigenvalues of generalized symmetric eigenvalue problems.
Miyajima et al.~\cite{MiyajimaOgitaRumpOishi2010} combined techniques developed for symmetric matrices in \cite{MiyajimaOgitaOishi2005TJSIAM, MiyajimaOgitaOishi2005, MiyajimaOgitaOishi2006TJSIAM} with Rump and Wilkinson's bounds to verify all eigenpairs of generalized symmetric eigenvalue problems.
Miyajima~\cite{Miyajima2012JCAM} improves his previous work for non-Hermitian $A$ and nonsingular non-Hermitian positive definite~$B$.
Miyajima~\cite{Miyajima2014SIMAX} uses Brouwer's fixed-point theorem for the enclosure of the eigenvalues and invariant subspaces of generalized non-Hermitian eigenvalue problems.

Our previous study proposes a verification method using complex moments~\cite{ImakuraMorikuniTakayasu2019JCAM}.
This method is based on an eigensolver~\cite{SakuraiSugiura2003JCAM}, which reduces a given generalized Hermitian eigenvalue problem into another generalized eigenvalue problem with block Hankel matrices, and evaluates all the errors in the reduction for verification.
We call this method the Hankel matrix approach throughout.
The errors are split into truncation errors in numerical quadrature and rounding errors.
To evaluate the truncation error, an interval arithmetic-friendly formula is derived.
This method is feasible even when $B$ is singular.
Also, we develop an efficient technique to validate the solutions of linear systems of equations corresponding to each quadrature point.

This study improves its truncation error using the Rayleigh--Ritz procedure~\cite{SakuraiTadano2007,IkegamiSakurai2010} and halves the number of quadrature points required by the Hankel matrix approach to satisfy a prescribed quadrature error.
This Rayleigh--Ritz procedure approach inherits features of the Hankel matrix approach, such as the efficient error evaluation technique for linear systems and the parameter tuning technique.
This approach is also feasible for singular $B$ when verifying eigenvalues and enables verifying eigenvectors.
Moreover, as this approach relies on the verification method for multiple and nearly multiple eigenvalues~\cite{Rump2001LAA}, it can deal with those eigenvalues in the interval~$\Omega$.
Numerical experiments prove the feasibility of this concept and show the efficiency and verification performance of the proposed method.

This paper is organized as follows.
Section~\ref{sec:proposed} presents the proposed method, derives computable error bounds for complex moments to justify it, and discusses implementation issues.
Section~\ref{sec:nuex} presents experimental results to illustrate the performance of the proposed method.
Section~\ref{sec:conc} concludes the paper.

\section{Rayleigh--Ritz procedure approach} \label{sec:proposed}
The Rayleigh--Ritz procedure projects a given eigenvalue problem into an (approximated) eigenspace of interest.
We develop a verified computation method for generalized Hermitian eigenvalue problems using the Rayleigh-Ritz procedure.
To this end, we first review a projection method using the Rayleigh--Ritz procedure and complex moments~\cite{SakuraiTadano2007,IkegamiSakurai2010}.

Define the $k$th complex moment matrix by
\begin{equation}
	M_k = \frac{1}{2 \pi \mathrm{i}} \oint_\Gamma \left(z-\gamma\right)^k (z B - A)^{-1} \mathrm{d} z, \quad k = 0, 1, 2, \dots, m-1	\label{eq:moment}
\end{equation}
on a positively oriented closed Jordan curve $\Gamma$ through the end points of the interval~$\Omega = [a, b]$, where $\mathrm{i}=\sqrt{-1}$ is the imaginary unit, $\pi$ is the circle ratio, and $\gamma \in \mathbb{R}$ is a shift parameter.
Then, using the matrix
\begin{equation}
	S = [S_0, S_1, \dots, S_{m-1}], \quad S_k = M_k B V , \quad k = 0, 1, 2, \dots, m - 1,
	\label{eq:Sk}
\end{equation}
where $V \in \mathbb{C}^{n \times \ell}$ is a random matrix, we transform the eigenvalue problem~\eqref{eq:gevp} into a reduced eigenvalue problem
\begin{align}
	S^\mathsf{H} ( A - \gamma B) S \boldsymbol{y} = (\lambda - \gamma) S^\mathsf{H} B S \boldsymbol{y}, \quad \boldsymbol{x} = S \boldsymbol{y}, \quad \boldsymbol{y} \in \mathbb{C}^n \setminus \lbrace \boldsymbol{0} \rbrace,
	\label{eq:RR}
\end{align}
where $\gamma \in \mathbb{R}$ is a shift parameter.
By solving the transformed generalized eigenvalue problem~\eqref{eq:RR}, we obtain the eigenvalues of interest under certain conditions.

We then show the identity between the Rayleigh--Ritz procedure approach and the Hankel matrix approach~\cite{SakuraiSugiura2003JCAM}.
To this end, we rewrite the coefficient matrices of \eqref{eq:RR} below.
Recall the Weierstrass canonical form of the matrix pencil $z B - A$ \cite[Proposition~7.8.3]{Bernstein2018}.
There exists a nonsingular matrix $X \in \mathbb{C}^{n \times n}$ such that
\begin{align}
	X^\mathsf{H} (zB - A) X = z \mathrm{I}_\mathrm{o} - \Lambda,
\end{align}
where the $i$th column of $X$ is the eigenvector~$\boldsymbol{x}_i$ corresponding to the eigenvalue~$\lambda_i$, $\mathrm{I}_\mathrm{o} = \mathrm{I}_r \oplus \mathrm{O} \in \mathbb{R}^{n \times n}$, and $\Lambda = \mathrm{diag}(\lambda_1, \lambda_2, \dots, \lambda_r) \oplus \mathrm{I}_{n - r} \in \mathbb{R}^{n \times n}$ whose leading $r$ diagonal entries are the eigenvalues of \eqref{eq:gevp}.
Here, $\mathrm{I}_t \in \mathbb{R}^{t \times t}$ is the identity matrix and $\oplus$ denotes the direct sum of matrices.
With this canonical form and the eigendecomposition
\begin{align}
	(z B - A)^{-1} 
	& = X (z \mathrm{I}_\mathrm{o} - \Lambda)^{-1} X^\mathsf{H} \\
	& = \sum_{i=1}^r (z - \lambda_i)^{-1} \boldsymbol{x}_i {\boldsymbol{x}_i}^\mathsf{H},
\end{align}
Caucy's integral formula gives the $k$th order complex moment
\begin{align}
	M_k & = \sum_{i=1}^r \left[ \frac{1}{2 \pi \mathrm{i}} \oint_\Gamma (z - \gamma)^k (z - \lambda_i)^{-1} \mathrm{d} z \right] \boldsymbol{x}_i {\boldsymbol{x}_i}^\mathsf{H} \\
	& = \sum_{i=1}^t (\lambda_i - \gamma)^k \boldsymbol{x}_i {\boldsymbol{x}_i}^\mathsf{H} \\
	& = X_\Omega {( \Lambda_{\Omega} - \gamma \mathrm{I}_t )}^k X_\Omega^\mathsf{H}
\end{align}
for $k = 0$, $1$, $\dots$, $m-1$, where $X_\Omega = \left[ \boldsymbol{x}_1, \boldsymbol{x}_2, \dots, \boldsymbol{x}_t \right]$ and $\Lambda_\Omega = \mathrm{diag} \left( \lambda_1, \lambda_2, \dots, \lambda_t \right)$.
Hence, we rewrite the coefficient matrices of \eqref{eq:RR} as
\begin{align*}
	{S_i}^\mathsf{H} (A - \gamma B) S_j
	& = V^\mathsf{H} B X_\Omega (\Lambda_\Omega - \gamma \mathrm{I}_t)^i [ X_\Omega^\mathsf{H} (A - \gamma B) X_\Omega ] {( \Lambda_{\Omega} - \gamma \mathrm{I}_t )}^j X_\Omega^\mathsf{H} B V \\
	& = V^\mathsf{H} B X_\Omega (\Lambda_\Omega - \gamma \mathrm{I}_t)^{i+j+1} X_\Omega^\mathsf{H} B V
\end{align*}
and
\begin{align*}
	{S_i}^\mathsf{H} B S_j
	& = V^\mathsf{H} B X_\Omega (\Lambda_\Omega - \gamma \mathrm{I}_t)^i ({X_\Omega}^\mathsf{H} B X_\Omega) (\Lambda_\Omega - \gamma \mathrm{I}_t)^j X_\Omega B V \\
	& = V^\mathsf{H} B X_\Omega (\Lambda_\Omega - \gamma \mathrm{I}_t)^{i+j} X_\Omega B V 
\end{align*}
for $i, j = 0$, $1$, $\dots$, $m-1$.
Here, we used the identity~${X_\Omega}^\mathsf{H} B X_\Omega = \mathrm{I}_t$, in which the eigenvectors~$\boldsymbol{x}_1$, $\boldsymbol{x}_2$, $\dots$, $\boldsymbol{x}_t$ are $B$-orthonormal.
Let $\mathsf{M}_k = V^\mathsf{H} B M_k B V$ be the reduced $k$th complex moment given in \cite[equation~(2)]{ImakuraMorikuniTakayasu2019JCAM}.
Then, the identities
\begin{align}
	{S_i}^\mathsf{H} (A - \gamma B) S_j = \mathsf{M}_{i+j+1}, \quad {S_i}^\mathsf{H} B S_j = \mathsf{M}_{i+j}
	\label{eq:SM}
\end{align}
for $i$, $j = 0$, $1$, $\dots$, $m-1$, or
\begin{align}
S^\mathsf{H} (A - \gamma B) S & = 
\begin{bmatrix}
\mathsf{M}_1 & \mathsf{M}_2 & \cdots & \mathsf{M}_m \\
\mathsf{M}_2 & \mathsf{M}_3 &  & \mathsf{M}_{m+1} \\
\vdots &  & \ddots & \vdots \\
\mathsf{M}_m & \mathsf{M}_{m+1} & \cdots & \mathsf{M}_{2m-1}
\end{bmatrix}, \\
S^\mathsf{H} B S & = 
\begin{bmatrix}
	\mathsf{M}_0 & \mathsf{M}_1 & \cdots & \mathsf{M}_{m-1} \\
	\mathsf{M}_1 & \mathsf{M}_2 &  & \mathsf{M}_m \\
	\vdots &  & \ddots & \vdots \\
	\mathsf{M}_{m-1} & \mathsf{M}_m & \cdots &\mathsf{M}_{2m-2}
\end{bmatrix}
\label{eq:HMsHM}
\end{align}
show that the Rayleigh--Ritz procedure and Hankel matrix approaches reduce the generalized eigenvalue problems~\eqref{eq:gevp} into the same eigenvalue problem with block Hankel matrices.
The left-hand sides of \eqref{eq:SM} form the transformed matrices in the Rayleigh--Ritz procedure approach, whereas the right-hand sides of \eqref{eq:SM} form the transformed matrices in the Hankel matrix approach.
We call these two approaches the complex moment approaches.
Further, the following theorem justifies that these methods determine the eigenvalues and eigenvectors of \eqref{eq:gevp}.

\begin{theorem}[{\cite[Theorem~7]{IkegamiSakuraiNagashima2010}, \cite[Theorem~3]{ImakuraDuSakurai2016JJIAM}}] \label{th:RR}
	Let $t$ be the number of eigenvalues of \eqref{eq:gevp} in the region~$\Omega$ and $S \in \mathbb{C}^{n \times \ell m}$ be defined as in \eqref{eq:Sk}, and assume $\mathrm{rank} S = t$.
	Then, the eigenvalues of the regular part of the matrix pencil $S^\mathsf{H} (A - z B) S$ are the same as the eigenvalues~$\lambda_i$ of \eqref{eq:gevp}, $i = 1$, $2$, $\dots$, $t$.
	Let $\boldsymbol{u}_i$ be the eigenvector corresponding to the eigenvalue~$\lambda_i$ of $S^\mathsf{H} (A - z B) S$.
	Then, $\boldsymbol{x}_i = S \boldsymbol{u}_i$ is the eigenvector corresponding to the eigenvalue~$\lambda_i$ of \eqref{eq:gevp}.
\end{theorem}

Note that this theorem holds even in the presence of multiple eigenvalues.
\medskip

The difference between the Rayleigh--Ritz and Hankel matrix approaches arises when approximating the integral~\eqref{eq:moment} using numerical quadrature.
Next, we evaluate the error in the Rayleigh--Ritz procedure approach, similarly to the previous study for the Hankel matrix  approach~\cite[sections~2, 3]{ImakuraMorikuniTakayasu2019JCAM}.

\subsection{\texorpdfstring{$N$}{N}-point quadrature rule.}
The complex moment~\eqref{eq:moment} is approximated by using the $N$-point trapezoidal rule, taking a circle with center~$\gamma$ and radius~$\rho$ in the complex plane
\begin{align}
	\Gamma = \left\{ z \in \mathbb{C} | z = \gamma + \rho \mathrm{exp}(\mathrm{i} \theta), \theta \in \mathbb{R} \right\}, \quad \gamma = \frac{b + a}{2}, \quad \rho = \frac{b - a}{2}
	\label{eq:integral_domain}
\end{align}	
as the domain of integration~$\Gamma$.
It follows from the error analysis in \cite{MiyataDuSogabeYamamotoZhang2009TJSIAM} that the $N$-point trapezoidal rule with the equi-distributed quadrature points
\begin{align}
	z_j = \gamma + \rho \mathrm{exp}(\mathrm{i} \theta_j), \quad \theta_j = \frac{2j-1}{N} \pi, \quad j = 1, 2, \dots, N
	\label{eq:zj_thetaj}
\end{align}
approximates the complex moment~$M_k$ as 
\begin{align}
M_k \simeq M_k^{(N)} = \sum_{i=1}^r (\lambda_i - \gamma)^k d_i^{(N)} \boldsymbol{x}_i \boldsymbol{x}_i^\mathsf{H},
\end{align}
where
\begin{align}
	d_i^{(N)} = 
	\begin{cases}
		\displaystyle \frac{1}{1 - \left( \frac{\lambda_i - \gamma}{\rho} \right)^N}, & i = 1, 2, \dots, t, \\
		\displaystyle \frac{- \left( \frac{\rho}{\lambda_i - \gamma} \right)^N}{1 - \left( \frac{\rho}{\lambda_i - \gamma} \right)^N}, & i = t+1, t+2, \dots, r.
	\end{cases}
	\label{eq:diN}
\end{align}
The approximation~$M_k \simeq M_k^{(N)}$ is confirmed as $d_i^{(N)} \rightarrow 1$ for $i = 1$, $2$, $\dots$, $t$ and $d_i ^{(N)} \rightarrow 0$ for $i = t+1$, $t+2$, $\dots$, $r$ for $N \rightarrow \infty$.

\subsection{Effect of eigenvalues inside and outside \texorpdfstring{$\Omega$}{Omega}}

To see the effect of the eigenvalues inside and outside the interval~$\Omega$ on the quadrature errors and for notational convenience, we split the complex moment into two
\begin{align}
	M_k^{(N)} = M_{k, \mathrm{in}}^{(N)} + M_{k, \mathrm{out}}^{(N)},
\end{align}
where 
\begin{align}
M_{k, \mathrm{in}}^{(N)} 
& = X_\Omega (\Lambda_\Omega - \gamma \mathrm{I}_t)^k D_\Omega^{(N)} {X_\Omega}^\mathsf{H},\\
M_{k, \mathrm{out}}^{(N)} 
& = X_{\Omega^\mathrm{c}} (\Lambda_{\Omega^\mathrm{c}} - \gamma \mathrm{I}_{r-t})^k D_{\Omega^\mathrm{c}}^{(N)} \label{eq:MkNout} {X_{\Omega^\mathrm{c}}}^\mathsf{H}
\end{align}
are associated with the eigenvalues inside and outside the interval~$\Omega$, respectively, for $k = 0$, $1$, $\dots$, $m-1$.
Here, we used the notations
\begin{align}
	D_\Omega^{(N)} & = \mathrm{diag}(d_1^{(N)}, d_2^{(N)}, \dots, d_t^{(N)}), \\
	D_{\Omega^\mathrm{c}}^{(N)} & = \mathrm{diag}(d_{t+1}^{(N)}, d_{t+2}^{(N)}, \dots, d_r^{(N)}), \\
	X_{\Omega^\mathrm{c}} & = [\boldsymbol{x}_{t+1}, \boldsymbol{x}_{t+1}, \dots, \boldsymbol{x}_r], \\
	\Lambda_{\Omega^\mathrm{c}} & = \mathrm{diag} (\lambda_{t+1}, \lambda_{t+2}, \dots, \lambda_r).
\end{align}

With the above approximation~$M_k \simeq M_k^{(N)}$, $k = 0, 1, \dots, 2m-1$, we obtain the approximated transformation matrix
\begin{align}
S_k	\simeq S_k^{(N)}= M_k^{(N)} B V
\end{align}
and split it into two $S_k^{(N)} = S_{k, \mathrm{in}}^{(N)} + S_{k, \mathrm{out}}^{(N)}$, where 
\begin{align}
	S_{k, \mathrm{in}}^{(N)} & = M_{k, \mathrm{in}}^{(N)} B V, 	\label{eq:SkinN} \\
	S_{k, \mathrm{out}}^{(N)} & = M_{k, \mathrm{out}}^{(N)} B V \label{eq:SkoutN}
\end{align}
are associated with the eigenvalues inside and outside the region~$\Omega$, respectively.
With this approximated transformation matrix~$S_k^{(N)}$, the reduced complex moment~$\mathsf{M}_{i+j+1}$ is approximated as 
\begin{align}
	\mathsf{M}_{i+j+1} 
	& \simeq \mathsf{M}_{i+j+1}^{(N)} \\
	& = (S_i^{(N)})^\mathsf{H} (A - \gamma B) S_j^{(N)}.
	\label{eq:Mkl1N}
\end{align}
The approximated reduced complex moment is split into two
\begin{align}
\mathsf{M}_{i+j+1}^{(N)} = \mathsf{M}_{i+j+1, \mathrm{in}}^{(N)} + \mathsf{M}_{i+j+1, \mathrm{out}}^{(N)},
\label{eq:splitMkN}
\end{align}
where 
\begin{align}
	\mathsf{M}_{i+j+1, \mathrm{in}}^{(N)} 
	& = (S_{i, \mathrm{in}}^{(N)})^\mathsf{H} (A - \gamma B) S_{j, \mathrm{in}}^{(N)}, \\
	\mathsf{M}_{i+j+1, \mathrm{out}}^{(N)} \label{eq:MijinN}
	& = (S_{i, \mathrm{out}}^{(N)})^\mathsf{H} (A - \gamma B) S_{j, \mathrm{out}}^{(N)}
\end{align}
are associated with the eigenvalues inside and outside the region~$\Omega$, respectively, for $i$, $j = 0$, $1$, $\dots$, $m-1$.

Let $H_m^< = S^\mathsf{H} (A - \gamma B) S$ and $H_m = S^\mathsf{H} B S$ be the block Hankel matrices in \eqref{eq:HMsHM}.
Note that the block $(i, j)$ entries of $H_m^<$ and $H_m$ are $\mathsf{M}_{i+j+1}$ and $\mathsf{M}_{i+j}$, respectively.
Then, in the Rayleigh--Ritz procedure approach, they are approximated as
\begin{align}
	H_m^< & \simeq H_m^{<, (N)} = (S^{(N)})^\mathsf{H} (A - \gamma B) S^{(N)}, \\
	H_m & \simeq H_m^{(N)} = (S^{(N)})^\mathsf{H} B S^{(N)},
\end{align}
where $S^{(N)} = [S_0^{(N)}, S_1^{(N)}, \dots, S_{M-1}^{(N)}]$.
Here, the block $(i, j)$ entries of $H_m^{<, (N)}$ and $H_m^{(N)}$ are $\mathsf{M}_{i+j+1}^{(N)}$ and $\mathsf{M}_{i+j}^{(N)}$, respectively.
For convenience, we split the approximated block Hankel matrices into two
\begin{align}
H_m^{<, (N)} = H_{m, \mathrm{in}}^{<, (N)} + H_{m, \mathrm{out}}^{<, (N)}, \quad H_m^{(N)} = H_{m, \mathrm{in}}^{(N)} + H_{m, \mathrm{out}}^{(N)},	
\end{align}
where
\begin{align}
H_{m, \mathrm{in}}^{<, (N)} = (S_\mathrm{in}^{(N)})^\mathsf{H} (A - \gamma B) S_\mathrm{in}^{(N)}, \quad H_{m, \mathrm{out}}^{<, (N)} = (S_\mathrm{out}^{(N)})^\mathsf{H} (A - \gamma B) S_\mathrm{out}^{(N)}
\label{eq:H_insN}
\end{align}
and
\begin{align}
H_{m, \mathrm{in}}^{(N)} = (S_\mathrm{in}^{(N)})^\mathsf{H} B S_\mathrm{in}^{(N)}, \quad H_{m, \mathrm{out}}^{(N)} = (S_\mathrm{out}^{(N)})^\mathsf{H} B S_\mathrm{out}^{(N)}
\label{eq:H_inN}
\end{align}
are associated with the eigenvalues inside and outside the region~$\Omega$, respectively,
Here, the block $(i, j)$ entries of $H_{m, \mathrm{in}}^{<, (N)}$, $H_{m, \mathrm{out}}^{<, (N)}$, $H_{m, \mathrm{in}}^{(N)}$, and $H_{m, \mathrm{out}}^{(N)}$ are $\mathsf{M}_{i+j+1, \mathrm{in}}^{(N)}$, $\mathsf{M}_{i+j+1, \mathrm{out}}^{(N)}$, $\mathsf{M}_{i+j, \mathrm{in}}^{(N)}$, and $\mathsf{M}_{i+j, \mathrm{out}}^{(N)}$.

\subsection{Verification of eigenvalues.}
To validate the eigenvalues of \eqref{eq:RR}, it is straightforward to enclose the coefficient matrices of \eqref{eq:RR}, i.e., \eqref{eq:HMsHM}.
Nevertheless, we exploit alternative quantities.
To this end, we prepare the following lemma.	

\begin{lemma} \label{lm:DXBX=XBXD}
	Let $D = D_1 \oplus D_2 \in \mathbb{R}^{n \times n}$ be a diagonal matrix with $D_1 \in \mathbb{R}^{t \times t}$ and the column vectors of $X \in \mathbb{C}^{n \times n}$ and $X_\Omega \in \mathbb{C}^{n \times t}$ be the eigenvectors~$\boldsymbol{x}_1$, $\boldsymbol{x}_2$, $\dots$, $\boldsymbol{x}_n$ and $\boldsymbol{x}_1$, $\boldsymbol{x}_2$, $\dots$, $\boldsymbol{x}_t$ of \eqref{eq:gevp}, respectively.
	Then, we have
	\begin{align}
		D_1 {X_\Omega}^\mathsf{H} B X = {X_\Omega}^\mathsf{H} B X D.
	\end{align}
\end{lemma}
\begin{proof}
As ${X_\Omega}^\mathsf{H} B X = [\mathrm{I}_t, \mathrm{O}]$ holds for the $B$-orthonormality of the eigenvectors, we have 
\begin{align}
D_1 {X_\Omega}^\mathsf{H} B X
& = D_1 [\mathrm{I}_t, \mathrm{O}] \\
& = [\mathrm{I}_t, \mathrm{O}] D \\
& = {X_\Omega}^\mathsf{H} B X D.
\end{align}
\end{proof}

We now give a link between the coefficient matrices of \eqref{eq:RR} and their splittings.

\begin{theorem} \label{th:same_eigenvalues}
Let $B$ be a Hermitian positive semidefinite matrix and $S$ be defined as in \eqref{eq:Sk} and 
\begin{align}
S_\mathrm{in}^{(N)} = \left[ S_{0, \mathrm{in}}^{(N)}, S_{1, \mathrm{in}}^{(N)}, \dots, S_{m-1, \mathrm{in}}^{(N)} \right],
\label{eq:SinN}
\end{align}
where $S_{k, \mathrm{in}}^{(N)}$ is as defined in \eqref{eq:SkinN}.
Assume $\mathrm{rank} S = t$.
Then, the matrix pencils~$S^\mathsf{H} (A - z B) S$ and $(S_\mathrm{in}^{(N)})^\mathsf{H} (A - z B) S_\mathrm{in}^{(N)}$ have the same eigenvalues.
\end{theorem}
\begin{proof}
Let $D^{(N)} = \mathrm{diag} (d_1^{(N)}, d_2^{(N)}, \dots, d_n^{(N)})$ with $d_i^{(N)} \in \mathbb{C}$ defined in \eqref{eq:diN} and $X \in \mathbb{C}^{n \times n}$ be defined as in Lemma~\ref{lm:DXBX=XBXD}.
Denote the $j$th column vector of $V = X C \in \mathbb{C}^{n \times \ell}$ and $V^{(N)} = X D^{(N)} C \in \mathbb{C}^{n \times \ell}$ by $\boldsymbol{v}_j = \sum_{i=1}^n c_{i j} \boldsymbol{x}_i$ and $\boldsymbol{v}_i^{(N)} = \sum_{i=1}^n c_{ij} d_i^{(N)} \boldsymbol{x}_i$, respectively, i.e., an expansion of the $j$th column of $V$ by the eigenvectors, for $j = 1$, $2$, $\dots$, $\ell$, where $C = (c_{ij}) \in \mathbb{C}^{n \times \ell}$.
Then, we have
\begin{align}
(S_{i, \mathrm{in}}^{(N)})^\mathsf{H} (A - \gamma B) S_{j, \mathrm{in}}^{(N)}
& = V^\mathsf{H} B X_\Omega D_\Omega^{(N)} (\Lambda_\Omega - z \mathrm{I}_t)^{i+j+1} D_\Omega^{(N)} {X_\Omega}^\mathsf{H} B V \\
& = (V^{(N)})^\mathsf{H} B X_\Omega (\Lambda_\Omega - z \mathrm{I}_t)^{i+j+1} {X_\Omega}^\mathsf{H} B V^{(N)}
\end{align}
for $i, j = 0, 1, \dots, m-1$.
From \eqref{eq:Sk} and $S_{k, \mathrm{in}} = X_\Omega (\Lambda_\Omega - \gamma \mathrm{I}_t)^k {X_\Omega}^\mathsf{H} B V^{(N)}$, it follows that we have the identity~$\mathrm{rank} (S) = \mathrm{rank} (S_\mathrm{in}^{(N)}) = t$.
Because Theorem~\ref{th:RR} holds even replacing $V$ by $V^{(N)}$, \eqref{eq:gevp} and $(S_\mathrm{in}^{(N)})^\mathsf{H} (A - z B) S_\mathrm{in}^{(N)}$ have the same eigenvalues.
Therefore, the assertion holds.
\end{proof}

Thanks to the relationships~\eqref{eq:MijinN} and $\mathsf{M}_{i+j, \mathrm{in}}^{(N)} = (S_{i, \mathrm{in}}^{(N)})^\mathsf{H} B S_{j, \mathrm{in}}^{(N)}$ and Theorem~\ref{th:same_eigenvalues}, we enclose $\mathsf{M}_{k, \mathrm{in}}^{(N)}$ instead of $\mathsf{M}_k$ for $k = 0, 1, \dots, 2 m-1$.
From the splitting~\eqref{eq:splitMkN}, $\mathsf{M}_{k, \mathrm{out}}^{(N)}$ can be regarded as the truncated error for quadrature.
Denote the quantity obtained by numerically computing $\mathsf{M}_k^{(N)}$ by $\tilde{\mathsf{M}}_k^{(N)}$.
Hereafter, we denote a numerically computed quantity that may suffer from rounding errors with a tilde.

\begin{theorem} \label{th:errbound}
	Denote the interval matrix with radius~$R \in \mathbb{R}_+^{\ell \times \ell}$ and center at $C \in \mathbb{R}^{\ell \times \ell}$ by $\langle C, R \rangle$.
	Then, the enclosure of $\mathsf{M}_{k, \mathrm{in}}^{(N)}$ is given by
	\begin{align}
		\mathsf{M}_{k, \mathrm{in}}^{(N)} & \in \left\langle \mathsf{M}_k^{(N)}, \left| \mathsf{M}_{k, \mathrm{out}}^{(N)} \right| \right\rangle \\
		& \subset \left\langle \tilde{\mathsf{M}}_k^{(N)}, \left| \mathsf{M}_{k, \mathrm{out}}^{(N)} \right| + \left| \tilde{\mathsf{M}}_k^{(N)} - \mathsf{M}_k^{(N)} \right| \right\rangle
		\label{eq:enclosure4MkinN}
	\end{align}
	for $k = 0$, $1$, $\dots$, $2m-1$.
\end{theorem}
\begin{proof}
The first enclosure of $\mathsf{M}_{k, \mathrm{in}}^{(N)}$ is obtained by the equality~$\mathsf{M}_k^{(N)} - \mathsf{M}_{k, \mathrm{in}} = \mathsf{M}_{k, \mathrm{out}}^{(N)}$ for $k = 0, 1, \dots, 2m-1$.
The second enclosure is obtained by using this equality and the inequality
\begin{align}
	\left| \mathsf{M}_{k, \mathrm{in}}^{(N)} - \tilde{\mathsf{M}}_{k}^{(N)} \right| & \leq \left| \mathsf{M}_{k, \mathrm{in}}^{(N)} - \mathsf{M}_k^{(N)} \right| + \left| \tilde{\mathsf{M}}_k^{(N)} - \mathsf{M}_k^{(N)} \right| \\
	& = \left| \mathsf{M}_{k, \mathrm{out}}^{(N)} \right| + \left| \tilde{\mathsf{M}}_k^{(N)} - \mathsf{M}_k^{(N)} \right|, \quad k = 0, 1, \dots, 2m-1.
\end{align}
\end{proof}

Theorem~\ref{th:errbound} implies that to enclose $\mathsf{M}_{k, \mathrm{in}}^{(N)}$, we can use $| \mathsf{M}_{k, \mathrm{out}}^{(N)} |$ and the truncated complex moment~$\mathsf{M}_k^{(N)}$ computed by using standard verification methods using interval arithmetic to obtain an enclosure of the truncation error~$| \mathsf{M}_k^{(N)} - \tilde{\mathsf{M}}_k^{(N)} |$.
Theorem~\ref{th:errbound} readily gives the following enclosure:
\begin{align}
H_{m, \mathrm{in}}^{<, (N)} & \subset \left\langle \tilde{H}_m^{<, (N)}, \left| H_{m, \mathrm{out}}^{<, (N)} \right| + \left| \tilde{H}_m^{<, (N)} - H_m^{<, (N)} \right| \right\rangle, \\
H_{m, \mathrm{in}}^{(N)} & \subset \left\langle \tilde{H}_m^{(N)}, \left| H_{m, \mathrm{out}}^{(N)} \right| + \left| \tilde{H}_m^{(N)} - H_m^{(N)} \right| \right\rangle.
\label{eq:HMinN}
\end{align}
An enclosure of $| \mathsf{M}_{k, \mathrm{out}}^{(N)} |$ is obtained as follows.

\begin{theorem}	\label{th:bound_M}
Let $B$ be a Hermitian positive semidefinite definite matrix.
Assume $2 m - 1 < N$ and that $\hat{\lambda} \in \mathbb{R}$ satisfies $| \hat{\lambda} - \gamma | = \min_{i=t+1, t+2, \dots, r} |\lambda_i-\gamma|$.
Then, $| \mathsf{M}_{k, \mathrm{out}}^{(N)} |$ in \eqref{eq:MkNout} is bounded by
\begin{align}
\left|\mathsf{M}_{k,\mathrm{out}}^{(N)}\right| \le (r-t)\left|\hat{\lambda}-\gamma\right|^k \left(\frac{\left(\frac{\rho}{\left|\hat{\lambda}-\gamma\right|}\right)^{2 N}}{1-\left(\frac{\rho}{\left|\hat{\lambda}-\gamma\right|}\right)^{2 N}}\right) \left\|V^{\mathsf{H}} B V \right\|_{\mathsf{F}}
\label{eq:bound_MkoutN}
\end{align}
for $k = 0$, $1$, $\dots$, $2m - 1$.
\end{theorem}

\begin{proof}
Let $\mathcal{V}_i = V^\mathsf{H} B \boldsymbol{x}_i \boldsymbol{x}_i^\mathsf{H} B V$.
Then,  applying the triangular inequality, we have
\begin{align}
\left| \mathsf{M}_{k, \mathrm{out}}^{(N)} \right| 
& = \left| \sum_{i=t+1}^r (\lambda_i - \gamma)^k {d_i}^2 \mathcal{V}_i \right| \\
& \leq \sum_{i=t+1}^r \left| \lambda_i - \gamma \right|^k {d_i}^2 \left| \mathcal{V}_i \right|
\end{align}
for $k = 0, 1, \dots, 2m-1$.
Noting the geometric series and applying the triangular inequality, we obtain
\begin{align}
{d_i}^2 
& =	\left[ \sum_{j=1}^\infty \left( \frac{\rho}{\lambda_i - \gamma} \right)^{jN} \right]^2 \\
& \leq \sum_{j=1}^\infty \left| \frac{\rho}{\lambda_i - \gamma} \right|^{2jN}
\end{align}
for $i = t+1$, $t+2$, $\dots$, $r$.
Multiplied by the factor~$| \lambda_i - \gamma |^k$, we obtain
\begin{align}
\left| \lambda_i - \gamma \right|^k {d_i}^2
& \leq \sum_{j=1}^\infty \rho^{2jN} \left| \lambda_i - \gamma \right|^{-(2jN-k)} \\
& \leq \sum_{j=1}^\infty \rho^{2jN} | \hat{\lambda} - \gamma |^{-(2jN-k)} \\
& = | \hat{\lambda} - \gamma |^k \frac{\left( \frac{\rho}{\left| \hat{\lambda} - \gamma \right|} \right)^{2N}}{1 - \left( \frac{\rho}{\left| \hat{\lambda} - \gamma \right|} \right)^{2N}}
\end{align}
for $i = t+1$, $t+2$, $\dots$, $r$ and $k = 0, 1, \dots, 2m-1$.
Here, the assumption~$2m-1 < N$ ensures $k < N$.
Noting that the last expression is independent of the index~$i$, we have
\begin{align}
\left|\mathsf{M}_{k,\mathrm{out}}^{(N)}\right| 
\leq | \hat{\lambda} - \gamma |^k \frac{\left( \frac{\rho}{\left| \hat{\lambda} - \gamma \right|} \right)^{2N}}{1 - \left( \frac{\rho}{\left| \hat{\lambda} - \gamma \right|} \right)^{2N}} \sum_{i=t+1}^r |\mathcal{V}_i |.
\end{align}
The bound $| \mathcal{V}_i | \leq \| V^\mathsf{H} B V \|_\mathsf{F}$ follows from the latter half of the proof of \cite[Theorem~3.3]{ImakuraMorikuniTakayasu2019JCAM}.
Therefore, we obtain \eqref{eq:bound_MkoutN}.
\end{proof}

\begin{remark}
The bound~\eqref{eq:bound_MkoutN} for the proposed Rayleigh--Ritz procedure approach is twice sharper than the one for the Hankel matrix approach~\cite[Theorem~3.3]{ImakuraMorikuniTakayasu2019JCAM}, i.e., the proposed method requires half the number of quadrature points required by the Hankel matrix approach to allow the same amount of truncation errors.
This observation is demonstrated in section~\ref{sec:nuex}.
\end{remark}

\subsection{Verification of eigenvectors.}
To verify the eigenvectors~$\boldsymbol{x}_i$ of \eqref{eq:gevp} via the Rayleigh--Ritz procedure approach as well as the Hankel matrix approach, we show the identity of the eigenvectors given by $S$ and $S_\mathrm{in}^{(N)}$.

\begin{theorem}
Assume that $B$ is a Hermitian and positive definite matrix.
Let $S$ and $S_\mathrm{in}^{(N)}$ be defined as in \eqref{eq:Sk} and \eqref{eq:SinN}, respectively, such that $\mathrm{rank} (S) = t$ and $\boldsymbol{y} \in \mathbb{C}^{\ell m}$ be an eigenvector of $(S_\mathrm{in}^{(N)})^\mathsf{H} A S_\mathrm{in}^{(N)} \boldsymbol{y} = \lambda (S_\mathrm{in}^{(N)})^\mathsf{H} B S_\mathrm{in}^{(N)} \boldsymbol{y}$.
	If $S \boldsymbol{y}$ is an eigenvector of \eqref{eq:gevp}, then $S_\mathrm{in}^{(N)} \boldsymbol{y}$ is also an eigenvector of \eqref{eq:gevp}.
\end{theorem}
\begin{proof}
Let $V^{(N)} = X D^{(N)} C$, where $X$ and $D^{(N)}$ are defined in Lemma~\ref{lm:DXBX=XBXD} and the proof of Theorem~\ref{th:same_eigenvalues}, respectively, and $C \in \mathbb{C}^{n \times L}$.
Then, from Lemma~\ref{lm:DXBX=XBXD}, it follows that
\begin{align}
S_{k, \mathrm{in}}^{(N)} 
& = X_\Omega (\Lambda_\Omega - \gamma \mathrm{I}_t)^k D_\Omega^{(N)} {X_\Omega}^\mathsf{H} B V \\
& = X_\Omega (\Lambda_\Omega - \gamma \mathrm{I}_t)^k {X_\Omega}^\mathsf{H} B V^{(N)}.
\end{align}
Because each eigencomponent of each column vector of $V^{(N)}$ is a scalar multiple of that of $V$, we have the identity~$\mathcal{R}(S_k) = \mathcal{R}(S_{k, \mathrm{in}}^{(N)})$, $k = 0, 1, \dots, M-1$.
\end{proof}

Motivated by this theorem, we focus on verifying $S_\mathrm{in}^{(N)}$, instead of $S$.

\begin{theorem}
Let 
\begin{align}
	S_\mathrm{out}^{(N)} = [S_{0, \mathrm{out}}^{(N)}, S_{1, \mathrm{out}}^{(N)}, S_{m-1, \mathrm{out}}^{(N)}].
	\label{eq:SoutN}
\end{align}
Then, we have the following enclosure of the approximated transformation matrix:
\begin{align}
	S_\mathrm{in}^{(N)} & \in \left\langle S^{(N)}, \left| S_\mathrm{out}^{(N)} \right| \right\rangle \\
	& \subset \left\langle \tilde{S}^{(N)}, \left| S_\mathrm{out}^{(N)} \right| + \left| \tilde{S}^{(N)} - S^{(N)} \right| \right\rangle.
	\label{eq:enclosure_S}
\end{align}
\end{theorem}
\begin{proof}
	The proof is given similarly to that of Theorem~\ref{th:errbound}.
\end{proof}

\begin{theorem}	\label{th:bound_S}
Assume that $B$ is a Hermitian and positive definite matrix.
Assume $2 m - 1 < N$ and that $\hat{\lambda} \in \mathbb{R}$ satisfies $| \hat{\lambda} - \gamma | = \min_{i=t+1, t+2, \dots, r} |\lambda_i - \gamma|$.
Then, $S_{k, \mathrm{out}}^{(N)}$ defined in \eqref{eq:SkoutN} is bounded as
\begin{align}
\left|S_{k, \mathrm{out}}^{(N)}\right| \le (n-t)\left|\hat{\lambda}-\gamma\right|^k \left(\frac{\left(\frac{\rho}{\left|\hat{\lambda}-\gamma\right|}\right)^N}{1-\left(\frac{\rho}{\left|\hat{\lambda}-\gamma\right|}\right)^N}\right) \left( \| B^{-1} \|_2 \| V^\mathsf{H} B V \|_\mathsf{F}\right)^{1/2}
\label{eq:outer_S}
\end{align}
for $k = 0$, $1$, $\dots$, $m - 1$.
\end{theorem}
\begin{proof}
	Similarly to the proof of Theorem~\ref{th:bound_M}, we have
	\begin{align}
		\left| S_{k, \mathrm{out}}^{(N)} \right| 
		& = \left| \sum_{i=t+1}^r (\lambda_i - \gamma)^k d_i B^{-1/2} B^{1/2} \boldsymbol{x}_i \boldsymbol{x}_i^\mathsf{H} B V \right| \\
		& \leq \sum_{i=t+1}^r |\lambda_i - \gamma|^k \frac{\left( \frac{\rho}{|\lambda_i - \gamma|} \right)^N}{1 - \left( \frac{\rho}{|\lambda_i - \gamma|} \right)^N} \left| B^{-1/2} \right| \left| B^{1/2} \boldsymbol{x}_i \right| \left| \boldsymbol{x}_i^\mathsf{H} B^{1/2} \right| \left| B^{1/2} V \right| \\
		& \leq \| B^{-1/2} \|_2 \| B^{1/2} V \|_2 \sum_{i=t+1}^r |\lambda_i - \gamma|^k \sum_{p=1}^\infty \left( \frac{\rho}{| \lambda_i - \gamma |} \right)^{pN} {\| B^{1/2} \boldsymbol{x}_i \|_2}^2 \\
		& = \| B^{-1} \|_2^{1/2} {\| V^\mathsf{H} B V \|_2}^{1/2} \sum_{i=t+1}^r \sum_{p=1}^\infty \rho^{pN} | \lambda_i - \gamma |^{-(pN-k)}  \\
		& \leq \left( \| B^{-1} \|_2 \| V^\mathsf{H} B V \|_2 \right)^{1/2}\sum_{i=t+1}^r \sum_{p=1}^\infty \rho^{pN} |\hat{\lambda} - \gamma|^{-(pN-k)} \\
		& = \left( \lambda_\mathrm{min}(B)^{-1} \| V^\mathsf{H} B V \|_\mathsf{F} \right)^{1/2} (r-t) |\hat{\lambda} - \gamma|^k \left( \frac{\left(\frac{\rho}{| \hat{\lambda} - \gamma|}\right)^N}{1 - \left( \frac{\rho}{|\hat{\lambda} - \gamma|} \right)^N} \right)
	\end{align}	
	for $i = t+1$, $t+2$, $\dots$, $r$.
	Here, we used the $B$-orthonormality of the eigenvectors~${\| B^{1/2} \boldsymbol{x}_i \|_2}^2 = {\boldsymbol{x}_i}^\mathsf{H} B \boldsymbol{x}_i = 1$.
\end{proof}

\begin{remark}
The evaluations~\eqref{eq:enclosure_S}, \eqref{eq:outer_S} can also be used for the Hankel matrix approach~\cite{ImakuraMorikuniTakayasu2019JCAM} for the evaluation of eigenvectors.
\end{remark}

\begin{remark}
	In Theorem~\ref{th:bound_S}, a Hermitian matrix $B$ is required to be positive definite for the verification of eigenvectors, contrarily to the verification of eigenvalues, cf.~Theorem~\ref{th:bound_M}.
\end{remark}

The evaluation of the numerical error~$| \tilde{S}^{(N)} - S^{(N)} |$ in \eqref{eq:enclosure_S}, i.e., $| \tilde{S}_k^{(N)} - S_k^{(N)} |$ for each $k = 0$, $1$, $\dots$, $m-1$, involves the error evaluation of the solution
\begin{align}
Y_j = (z_j B - A)^{-1} B V	
\label{eq:Yj}
\end{align}
of the linear system of equations with multiple right-hand sides~$(z_j B - A) Y_j = B V$ associated with
\begin{align}
	S_k^{(N)} = \frac{1}{N} \sum_{j=1}^N \exp ((k+1) \theta_j \mathrm{i}) Y_j
	\label{eq:SkN}
\end{align}
for $k = 0$, $1$, $\dots$, $m-1$.
The enclosure of $Y_j$ can be obtained by using standard verification methods, e.g.~\cite{Rump2013JCAMa,Rump2013JCAMb}.
For efficiency, the technique based on \cite[Theorem~4.1]{ImakuraMorikuniTakayasu2019JCAM} can be also used.

\subsection{Implementation}
We present implementation issues of the proposed method.
We assume that the numbers of $\ell$ and $m$ satisfy $\ell m = t$.
Also, the proposed method needs to determine the number of the parameter~$N$.
Each quadrature point~$z_j$ gives rise to a linear system~$(z_j B - A) Y_j = B V$ to solve.
The evaluation of a solution for each linear system is the most expensive part, whereas the quadrature errors~$|\mathsf{M}_\mathrm{out}^{(N)}|$ and $|S_\mathrm{out}^{(N)}|$ reduce as the number of quadrature points~$N$ increases (see Theorems~\ref{th:bound_M} and \ref{th:bound_S}).
To achieve efficient verification, it is favorable to evaluate solutions of the linear systems as few as possible.
Hence, there is a trade-off between the computational cost and quadrature error.
The number of quadrature points~$N$ has been heuristically determined in the complex moment eivensolvers for numerical computations.
For numerical verification, a reasonable number~$N$ can be determined according to the quadrature error.
The error bounds~\eqref{eq:bound_MkoutN} and \eqref{eq:outer_S} can be used to determine a reasonable number of quadrature points.
The least number of $N$ such that 
\begin{empheq}[left={N \geq \empheqlbrace}]{alignat=2}
\frac{1}{2} \left( \log \frac{\rho}{|\hat{\lambda}-\gamma|} \right)^{-1} \log \left(\frac{\delta}{c_1 (r-t) + \delta}\right) & \quad \text{for eigenvalues}, \label{eq:N4eigval}\\
\left( \log \frac{\rho}{|\hat{\lambda}-\gamma|} \right)^{-1} \log \left(\frac{\delta}{c_2 (r-t) + \delta}\right) & \quad \text{for eigevectors} \label{eq:N4eigvec}
\end{empheq}
yields a quadrature error less than $\delta$, i.e., $\left|\mathsf{M}_{k,\mathrm{out}}^{(N)}\right| \leq \delta$ and $\left|S_{k, \mathrm{out}}^{(N)}\right| \leq \delta$, respectively, at the least cost, where
\begin{align}
	c_1 & =  \| V^\mathsf{H} B V \|_\mathsf{F} \max_{k=0, 1, \dots, 2m-1} |\hat{\lambda} - \gamma |^k, \\
	c_2 & = \left( \| B^{-1} \|_2 \| V^\mathsf{H} B V \|_\mathsf{F} \right)^{1/2} \max_{k=0, 1, \dots, m-1} |\hat{\lambda} - \gamma |^k.
\end{align}
We summarize the above procedures in Algorithm~\ref{alg:verifiedSSRR}.
Here, we denote interval quantities with square brackets~$[\cdot]$ and the quantity in the right-hand sides of \eqref{eq:bound_MkoutN} and \eqref{eq:outer_S} by $\mu_{k, \mathrm{out}}^{(N)}$ and $\sigma_{k, \mathrm{out}}^{(N)}$, respectively.
The computation of line~3 of Algorithm~\ref{alg:verifiedSSRR} can be performed as follows~\cite[p.~7]{ImakuraMorikuniTakayasu2019JCAM}:
\begin{enumerate}
	\item Compute a numerical approximation~$\tilde{\lambda}$ of $\hat{\lambda}$ such that $|\tilde{\lambda} - \gamma| > \rho$, defined in Theorems~\ref{th:bound_M} and \ref{th:bound_S}.
	
	\item Set $c \in (0, 1)$ such that $\rho < c |\tilde{\lambda}-\gamma|$.	
	
	\item Verify regularity of the interval matrices~$A - [\gamma+\rho, \gamma+c |\tilde{\lambda}-\gamma|]B$ and $A - [\gamma-  c |\tilde{\lambda}-\gamma|, \gamma-\rho]B$, e.g., by using the INTLAB function~\texttt{isregular}.
	
	\item Adopt $c|\tilde{\lambda} - \gamma|$ as a lower bound of $|\hat{\lambda} - \gamma|$.
\end{enumerate}
To choose a possible large value of $c$, Steps~2--3 can be performed by using a bisection method.

\begin{algorithm}[t]
	\caption{Rayleigh--Ritz procedure approach.}
	\label{alg:verifiedSSRR}
	\begin{algorithmic}[1]
		\REQUIRE $A \in \mathbb{C}^{n \times n}$, $B \in \mathbb{C}^{n \times n}$, $\ell$, $m \in \mathbb{N}_+$ such that $t = \ell m$, $V \in \mathbb{C}^{n \times \ell}$, $\gamma$, $\rho \in \mathbb{R}$, and $\delta>0$.
		\ENSURE $[\lambda_i]$, $[\boldsymbol{x}_i ]$, $i = 1$, $2$, $\dots$, $t$
		\STATE Determine $N$ by using \eqref{eq:N4eigval} or \eqref{eq:N4eigvec}.
		\STATE Compute $[\theta_j] = [(2 j - 1)\pi / N ]$, $[z_j] = [\gamma + \rho \exp(\mathrm{i} [\theta_j])]$, $j = 1$, $2$, $\dots$, $N$
		\STATE Rigorously compute a lower bound of $|\hat{\lambda} - \gamma |=\min_{k=t+1,t+2, \dots,r}\left|\lambda_k - \gamma\right|$.
		\STATE Compute $[| \mathsf{M}_{k,\mathrm{out}}^{(N)} |]$ from $[\mu_{k, \mathrm{out}}^{(N)}]$, $k = 0$, $1$, $\dots$, $2m-1$.
		\STATE Compute $[Y_j]$ in \eqref{eq:Yj}, $j = 1$, $2$, $\dots$, $N$.
		\STATE Compute $[S_k^{(N)}] = \left\langle \tilde{S}_k^{(N)}, | \tilde{S}_k^{(N)} - S_k^{(N)} | \right\rangle$ using \eqref{eq:SkN}, $k = 0, 1, \dots, m-1$.
		\STATE Compute $[\mathsf{M}_k^{(N)}] = \left\langle \tilde{\mathsf{M}}_k^{(N)}, | \tilde{\mathsf{M}}_k^{(N)} - \mathsf{M}_k^{(N)} | \right\rangle$ using \eqref{eq:Mkl1N}, $k = 0, 1, \dots, 2m-1$.
		\STATE Compute $[ \mathsf{M}_{k,\mathrm{in}}^{(N)}]$ using \eqref{eq:enclosure4MkinN}, $k = 0$, $1$, $\dots$, $2m-1$.
		\STATE From $[ H_{m, \mathrm{in}}^{<, (N)} ]$, $[ H_{m, \mathrm{in}}^{(N)} ]$ from $\mathsf{M}_{k, \mathrm{in}}^{(N)}$, $k = 0, 1, \dots, 2m-1$.
		\STATE Compute the eigenvalue~$[\lambda_i]$ and eigenvector~$[\boldsymbol{y}_i]$ of the generalized eigenvalue problem~$[ H_{m, \mathrm{in}}^{<, (N)} ] \boldsymbol{y}_i = \lambda_i [ H_{m, \mathrm{in}}^{<, (N)} ] \boldsymbol{y}_i$, $i = 1, 2, \dots, t$.
		\STATE Rigorously compute an upper bound of $| S_{k, \mathrm{out}}^{(N)} |$ using $\sigma_{k, \mathrm{out}}^{(N)}$, $k = 0$, $1$, $\dots$, $m-1$.
		\STATE Form $| S_\mathrm{out}^{(N)} |$ using \eqref{eq:SoutN}, $\tilde{S}^{(N)} = [\tilde{S}_0^{(N)}, \tilde{S}_1^{(N)}, \dots, \tilde{S}_{M-1}^{(N)}]$, and $| \tilde{S}^{(N)} - S^{(N)}|$.
		\STATE Compute $[S_\mathrm{in}^{(N)}]$ using \eqref{eq:enclosure_S}.
		\STATE Compute $[\boldsymbol{x}_i] = [ S_\mathrm{in}^{(N)} \boldsymbol{y}_i ]$.
	\end{algorithmic}
\end{algorithm}

\section{Numerical experiments} \label{sec:nuex}
Numerical experiments show that the proposed method is superior to previous methods in terms of efficiency, while maintaining verification performance.
The efficiency is evaluated in terms of CPU time.
The performance of verification is evaluated in terms of the radii of the intervals of the verified eigenvalue and entries of the eigenvectors.

All computations are performed on a computer with an Intel Xeon Platinum 8176M 2.10 GHz central processing unit (CPU), 3 TB of random-access memory (RAM), and the Ubuntu 18.04.5 LTS operating system.
All programs are implemented and run in MATLAB Version~9.6.0.1335978 (R2019a) Update~8 for double precision floating-point arithmetic with unit roundoff~$u = 2^{-53} \simeq 1.1\cdot10^{-16}$.
We use INTLAB version~11~\cite{Rump1999} for interval arithmetic.
The compared methods are the combination of the MATLAB built-in function~\texttt{eigs} for the solution of the eigenvalue problem and INTLAB function~\texttt{verifyeig} for verification, which is denoted by \texttt{eigs}+\texttt{verifyeig}, and the Hankel matrix approaches in \cite{ImakuraMorikuniTakayasu2019JCAM}.
The matrix~$V \in \mathbb{R}^{n \times \ell}$ are generated by using the built-in MATLAB function~\texttt{randn}.
The tolerance of quadrature error~$\delta$ is set to $10^{-15}$.
The number of quadrature points~$N$ for the complex moment approaches is determined according to the criteria~\cite[(12)]{ImakuraMorikuniTakayasu2019JCAM} and \eqref{eq:N4eigval} for the verification of eigenvalues and \eqref{eq:N4eigvec} for the verification of eigenvectors.
The eigenvalues and eigenvectors of $[ H_{m, \mathrm{in}}^{<, (N)} ] \boldsymbol{y} = \lambda [H_{m, \mathrm{in}}^{(N)}] \boldsymbol{y}$ in line~9 of Algorithm~\ref{alg:verifiedSSRR} are verified by using the INTLAB function~\texttt{verifyeig}.
Here, \texttt{verifyeig} can deal with multiple and nearly multiple eigenvalues~\cite{Rump2001LAA}.
Note again that the number of eigenvalues in the interval~$\Omega$ is assumed to be given in advance.

\subsection{Efficiency}
To show an advantage of the proposed method in efficiency in terms of the CPU time,  we test on the problem with matrices
\begin{align}
	A = \mathrm{tridiag} (-1, 2, -1) \in \mathbb{R}^{n \times n}, \quad B = \mathrm{diag} (b_1, b_2, \dots, b_n) \in \mathbb{R}^{n \times n} \label{eq:tridiag}
\end{align}
with size~$n = 2^s$, $s = 5, 6, \dots, 16$, where $\mathrm{tridiag} (\cdot, \cdot, \cdot)$ denotes the tridiagonal Toeplitz matrix consisting of a triplet and the value of $b_i$ normally distributes with mean $1$ and variance $10^{-7}$.
The eigenvalue problem with the coefficient matrices~\eqref{eq:tridiag} models an one-dimensional harmonic oscillator consisting of $n$ mass points and $n+1$ springs.
See \cite[section~5]{ImakuraMorikuniTakayasu2019JCAM} for details.

\begin{figure}[t]
	\centering
	\begin{minipage}{0.48\columnwidth}
		\centering
		\includegraphics[width=0.99\linewidth]{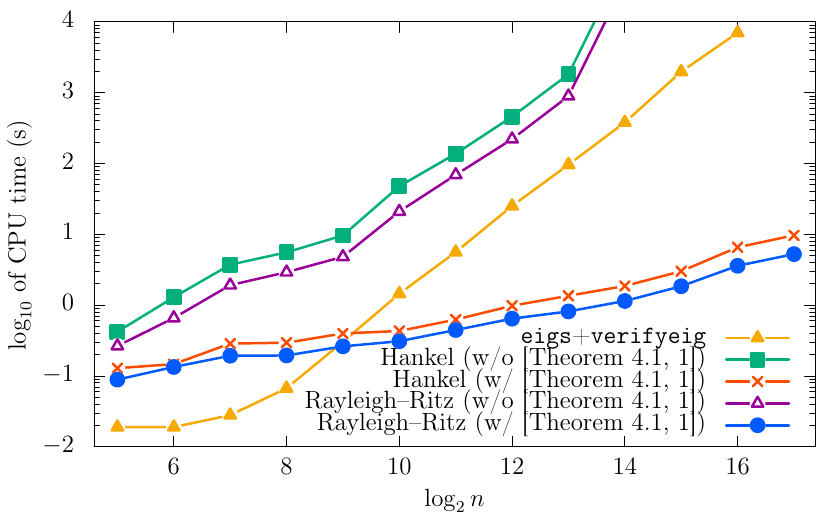}
		\caption{CPU time for each method versus the size~$n$ of the test problems with \eqref{eq:tridiag}.}
		\label{fig:N_vs_time}
	\end{minipage}
	\quad		
	\begin{minipage}{0.48\columnwidth}	
		\centering
		\includegraphics[width=0.99\linewidth]{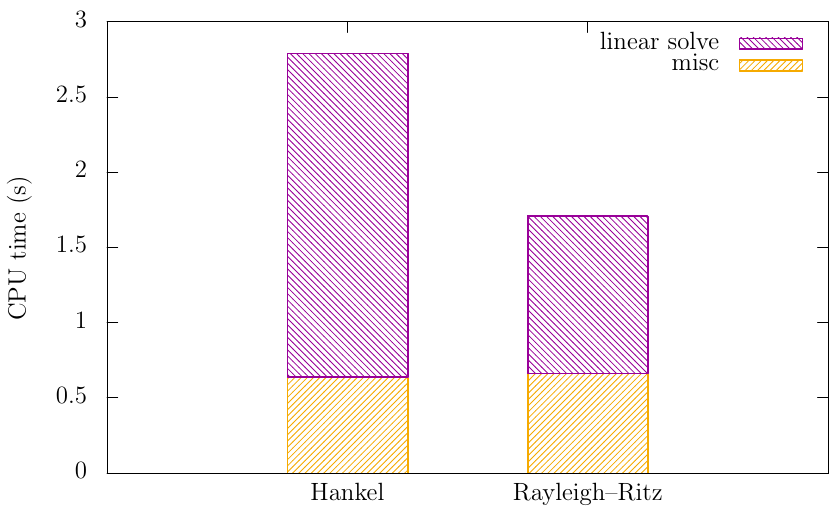}
		\caption{Breakdown of the CPU times for the Hankel matrix and Rayleigh--Ritz procedure approaches for the problem with \eqref{eq:tridiag} for $n = 2^{16}$.}
		\label{fig:stacked}
	\end{minipage}
\end{figure}

We compute and verify the four eigenvalues closest to two on the real axis so that we set the numbers of parameters~$\ell = m = 2$, i.e., $\ell m = 4$, and the contour~$\Gamma$ to a circle with center~$2$ on the real axis.
Perturbation theory of generalized Hermitian eigenvalue problems~\cite[Theorem~8.3]{Nakatsukasa2011th} gives the following bound between an eigenvalue $\lambda_i$ of \eqref{eq:gevp} and an eigenvalue $\lambda_i(A)$ of $A$:
\begin{align}
|\lambda_i (A) - \lambda_i| \leq |\lambda_i (A)| \| \Delta B \|_2 \| B^{-1} \|_2,
\end{align}
where $\Delta B = \mathrm{I} - B$.
Thus, a lower bound of $| \hat{\lambda} - \gamma |$ and radius~$\rho$ of $\Gamma$ are derived to enclose the four eigenvalues.

Figure~\ref{fig:N_vs_time} shows the elapsed CPU time for the proposed and compared methods versus the size of matrix pencils~\eqref{eq:tridiag}.
The Hankel matrix and Rayleigh--Ritz procedure approaches are tested when they use and do not use the technique based on \cite[Theorem~4.1]{ImakuraMorikuniTakayasu2019JCAM} for efficiently verifying the linear solve~\eqref{eq:Yj}.
The input arguments of \texttt{eigs} are set to compute the four eigenvalues closest to two on the real axis.
This figure shows that this technique substantially improves the efficiencies of these approaches in terms of the CPU time.
These approaches become faster than \texttt{eigs}+\texttt{verifyeig} for large cases with $s > 10$ and this is pronounced, as the matrix becomes large.
Further, the Rayleigh--Ritz procedure approach is nearly twice faster than the Hankel matrix approach.

To make a detailed observation, Figure~\ref{fig:stacked} shows the breakdown of the CPU times for the Hankel matrix and Rayleigh--Ritz procedure approaches using technique based on \cite[Theorem~4.1]{ImakuraMorikuniTakayasu2019JCAM} for the problem with \eqref{eq:tridiag} for $n = 2^{16}$.
The linear solve stands for the CPU time required to evaluate the linear solves with respect to the quadrature points, e.g., step~5 of Algorithm~\ref{alg:verifiedSSRR} for the Rayleigh--Ritz procedure approach.
The misc stands for the CPU time required for the other procedures.
This figure shows that the linear solve took more than half of the total CPU time and the Rayleigh--Ritz procedure approach results in twice faster CPU time than the Hankel matrix approach.
Since the $N$ linear solves can be evaluated independently, these approach can reduce the CPU time for the linear solves nearly to $1/N$ when implemented in parallel.
Evaluations of these performances in a parallel computer are left for future work.

Table~\ref{tbl:vev_art} gives the infimum and supremum of the verified eigenvalues for each number of $s$ for each method.
Here, the Hankel matrix and Rayleigh--Ritz procedure approaches employ the technique using \cite[Theorem~4.1]{ImakuraMorikuniTakayasu2019JCAM} for efficiently verifying the solution of the linear systems~\eqref{eq:Yj}.
Each row shows for each number of $s$, the infimum and supremum of the verified eigenvalues~$\lambda_1 \leq \lambda_2 \leq \lambda_3 \leq \lambda_4$.
In each subtable, each row gives digits that are the same as those of the exact eigenvalues in a single line and digits that mean the supremum and infimum of the exact eigenvalues in double lines.
The number of quadrature points~$N$ for the complex moment approaches is given in the second column.
These tables show that as $s$ increases, the number of correct digits tends to decrease and the required number of quadrature points tends to increase for the complex moment approaches.
The Hankel matrix approach tend to give more correct digits than the Rayleigh--Ritz approach.
Even as $s$ increases, \texttt{eigs}+\texttt{verifyeig} gives almost fully correct digits.
The Rayleigh--Ritz procedure approach requires half the number of quadrature points for the Hankel matrix approach.

Table~\ref{tbl:tridiag_eigvec} gives the maximum of the verified radii of the entries of the eigenvectors corresponding to the eigenvalues near $2$ for the test problems with \eqref{eq:tridiag}.
Here, the Hankel matrix and Rayleigh--Ritz procedure approaches employ the technique using \cite[Theorem~4.1]{ImakuraMorikuniTakayasu2019JCAM} for efficiently verifying the solution of the linear systems~\eqref{eq:Yj}.
In each subtable, each column shows for each number of $s$, the radius of the eigenvectors~$\boldsymbol{x}_1, \boldsymbol{x}_2, \boldsymbol{x}_3$, and $\boldsymbol{x}_4$ corresponding to $\lambda_1$, $\lambda_2$, $\lambda_3$, and $\lambda_4$.
These tables show that as the size~$n = 2^s$ of the problem increases, the maximum interval radius tends to increase.

\begin{remark}
In the above observations, the Rayleigh--Ritz procedure approach tends to give larger interval radii than the others.
A reason for this deterioration is that the enclosure of $M_{k, \mathrm{in}}^{(N)}$ is obtained from $|\mathsf{M}_k^{(N)} - \tilde{\mathsf{M}}_k^{(N)}|$ due to \eqref{eq:enclosure4MkinN}.
The latter is computed by \eqref{eq:Mkl1N}, which suffers rounding errors occurring in the solution~$Y_j$.
The enclosures of both~$\tilde{S}_i^{(N)}$ and $\tilde{S}_j^{(N)}$ affect the enclosure of $\tilde{\mathsf{M}}_{i+j+1}^{(N)}$.
This leads to increases of the interval radii of the coefficient matrices of the reduced eigenvalue problem~$[ H_{m, \mathrm{in}}^{<, (N)} ] \boldsymbol{y} = \lambda [H_{m, \mathrm{in}}^{(N)}] \boldsymbol{y}$ and the interval redii of the verified eigenpairs, as a by-product.
A remedy for improving the accuracy of the solution is to use iterative refinements~\cite{OishiOgitaRump2009}.
Meanwhile, the Hankel matrix approach suffers rounding errors in the computation of single complex moments.
Note that the truncation errors of quadrature for both complex moment approaches are in the same order in this experiment.
\end{remark}

\begin{table}[!ht]
	\caption{Infimum and supremum of the verified four eigenvalues near $2$ for the test problems with \eqref{eq:tridiag}.}
	\label{tbl:vev_art}
	\centering
	\begin{subtable}{\textwidth}
		\caption{\texttt{eigs}+\texttt{verifyeig}.}
		\scriptsize
		\centering
		\begin{tabular}{c|llll}
			$s$ & \multicolumn{1}{c}{$\lambda_1$} & \multicolumn{1}{c}{$\lambda_2$} & \multicolumn{1}{c}{$\lambda_3$} & \multicolumn{1}{c}{$\lambda_4$} \\
			\hline
			5 & $1.71537033235013_{4}^{5}$ & $1.90483618825803_{5}^{6}$ & $2.095163853542021$ & $2.28462968839571_{2}^{3}$ \\[1mm]
			6 & $1.855130417880838$ & $1.951672558144470$ & $2.04832754149673_{6}^{7}$ & $2.14486962137112_{3}^{4}$ \\[1mm]
			7 & $1.92695596486600_{8}^{9}$ & $1.97564720057065_{6}^{7}$ & $2.02435286052257_{4}^{5}$ & $2.07304405792919_{0}^{1}$ \\[1mm]
			8 & $1.963329774559972$ & $1.98777597272563_{8}^{9}$ & $2.01222401227200_{1}^{2}$ & $2.03667024056232_{6}^{7}$ \\[1mm]
			9 & $1.98162836857217_{6}^{7}$ & $1.99387604508410_{3}^{4}$ & $2.00612395050725_{4}^{5}$ & $2.01837162549658_{7}^{8}$ \\[1mm]
			10 & $1.990805123444199$ & $1.99693503831713_{8}^{9}$ & $2.00306497281993_{5}^{6}$ & $2.00919486928309_{4}^{5}$ \\[1mm]
			11 & $1.9954003070059_{79}^{80}$ & $1.99846676513378_{4}^{5}$ & $2.00153322910796_{7}^{8}$ & $2.00459969172964_{8}^{9}$ \\[1mm]
			12 & $1.99769958798105_{7}^{8}$ & $1.9992332015574_{19}^{20}$ & $2.00076680789580_{2}^{3}$ & $2.00230040607660_{5}^{6}$ \\[1mm]
			13 & $1.99884965651098_{8}^{9}$ & $1.99961655578429_{2}^{3}$ & $2.000383452559825$ & $2.0011503467218_{89}^{90}$ \\[1mm]
			14 & $1.999424794537495$ & $1.99980826342074_{8}^{9}$ & $2.000191735211724$ & $2.00057520989805_{7}^{8}$ \\[1mm]
			15 & $1.99971238742483_{2}^{3}$ & $1.999904128280850$ & $2.00009587002768_{6}^{7}$ & $2.00028761266374_{7}^{8}$ \\[1mm]
			16 & $1.999856191688560$ & $1.999952062997465$ & $2.00004793533377_{7}^{8}$ & $2.00014380869740_{0}^{1}$ \\
		\end{tabular}
	\end{subtable}
	\medskip
	
	\begin{subtable}{\textwidth}
		\caption{Hankel matrix approach.}
		\scriptsize
		\centering
		\begin{tabular}{c|cllll}		
			$s$ & $N$ & \multicolumn{1}{c}{$\lambda_1$} & \multicolumn{1}{c}{$\lambda_2$} & \multicolumn{1}{c}{$\lambda_3$} & \multicolumn{1}{c}{$\lambda_4$} \\
			\hline		
			5 & 76 & $1.7153703323_{47914}^{52348}$ & $1.9048361882_{47090}^{68985}$ & $2.0951638535_{33456}^{50593}$ & $2.284629688_{387397}^{404012}$ \\[1mm]
			6 & 78 & $1.8551304178_{76541}^{85160}$ & $1.951672558_{097438}^{193924}$ & $2.048327541_{394587}^{603956}$ & $2.1448696213_{48672}^{93821}$ \\[1mm]
			7 & 78 & $1.9269559648_{61249}^{70775}$ & $1.975647200_{456114}^{685677}$ & $2.024352860_{495618}^{550165}$ & $2.073044057_{888081}^{970314}$ \\[1mm]
			8 & 80 & $1.9633297745_{36408}^{83632}$ & $1.9877759727_{20604}^{30672}$ & $2.0122240122_{63992}^{80092}$ & $2.036670240_{510471}^{614189}$ \\[1mm]
			9 & 82 & $1.981628368_{531830}^{612764}$ & $1.99387604508_{0758}^{7442}$ & $2.006123950_{499947}^{514671}$ & $2.018371625_{469633}^{523545}$ \\[1mm]
			10 & 82 & $1.9908051234_{10658}^{77897}$ & $1.99693503_{7920041}^{8714281}$ & $2.003064972_{699502}^{941428}$ & $2.009194869_{235890}^{330294}$ \\[1mm]
			11 & 84 & $1.99540030_{6918166}^{7094033}$ & $1.99846676_{4963604}^{5303971}$ & $2.001533229_{073259}^{142923}$ & $2.00459969_{1350860}^{2108434}$ \\[1mm]
			12 & 86 & $1.9976995_{82687250}^{93275413}$ & $1.99923320_{0602390}^{2514280}$ & $2.00076680_{7032535}^{8765280}$ & $2.00230040_{2673471}^{9504552}$ \\[1mm]
			13 & 88 & $1.99884965_{3826925}^{9203680}$ & $1.99961655_{4098269}^{7479218}$ & $2.0003834_{48671582}^{56458023}$ & $2.0011503_{39850697}^{53592375}$ \\[1mm]
			14 & 88 & $1.9994247_{89407008}^{99668297}$ & $1.99980826_{2638995}^{4202715}$ & $2.00019173_{4398198}^{6025405}$ & $2.0005752_{08927021}^{10869043}$ \\[1mm]
			15 & 90 & $1.99971238_{5326499}^{9523053}$ & $1.9999041_{23444278}^{33118563}$ & $2.0000958_{68911830}^{71143737}$ & $2.00028761_{0469643}^{4857982}$ \\[1mm]
			16 & 92 & $1.999856_{166646706}^{216719115}$ & $1.99995206_{0260821}^{5734205}$ & $2.00004793_{1154953}^{9511890}$ & $2.0001438_{02072102}^{15323321}$ \\
		\end{tabular}
	\end{subtable}
	\medskip
	
	\begin{subtable}{\textwidth}
		\caption{Rayleigh--Ritz procedure approach.}
		\scriptsize
		\centering
		\begin{tabular}{c|cllll}
			$s$ & $N$ & \multicolumn{1}{c}{$\lambda_1$} & \multicolumn{1}{c}{$\lambda_2$} & \multicolumn{1}{c}{$\lambda_3$} & \multicolumn{1}{c}{$\lambda_4$} \\
			\hline	
			5 & 38 & $1.7153703323_{45711}^{54559}$ & $1.9048361882_{42155}^{73914}$ & $2.0951638535_{28928}^{55114}$ & $2.284629688_{377201}^{414223}$ \\[1mm]
			6 & 40 & $1.8551304178_{68328}^{93347}$ & $1.951672558_{045605}^{243334}$ & $2.048327541_{234644}^{758824}$ & $2.144869621_{272423}^{469818}$ \\[1mm]
			7 & 40 & $1.9269559648_{44142}^{87875}$ & $1.975647200_{268550}^{872767}$ & $2.024352860_{457115}^{588034}$ & $2.07304405_{7758953}^{8099429}$ \\[1mm]
			8 & 40 & $1.963329774_{431743}^{688207}$ & $1.987775972_{681278}^{769999}$ & $2.012224012_{233492}^{310514}$ & $2.03667024_{0113744}^{1010910}$ \\[1mm]
			9 & 42 & $1.98162836_{7744105}^{9400253}$ & $1.993876045_{024694}^{143511}$ & $2.006123950_{366358}^{648153}$ & $2.01837162_{4981080}^{6012101}$ \\[1mm]
			10 & 42 & $1.99080512_{2384967}^{4503434}$ & $1.9969350_{31603472}^{45030815}$ & $2.00306497_{0577924}^{5061955}$ & $2.0091948_{68359327}^{70206860}$ \\[1mm]
			11 & 42 & $1.99540030_{4355702}^{9656255}$ & $1.99846676_{0585749}^{9681816}$ & $2.0015332_{28163114}^{30052822}$ & $2.004599_{681042095}^{702417197}$ \\[1mm]
			12 & 44 & $1.997699_{233631056}^{942331091}$ & $1.999233_{151035052}^{252079796}$ & $2.000766_{763701373}^{852090242}$ & $2.002300_{208338236}^{603815044}$ \\[1mm]
			13 & 44 & $1.998849_{403804694}^{909217316}$ & $1.999616_{448114726}^{663453885}$ & $2.000383_{184381851}^{720737856}$ & $2.0011_{49721291246}^{50972152649}$ \\[1mm]
			14 & 44 & $1.99942_{4430921692}^{5158153283}$ & $1.999808_{225050437}^{301791060}$ & $2.000191_{686133423}^{784290023}$ & $2.000575_{100797855}^{318998258}$ \\[1mm]
			15 & 46 & $1.999712_{114463508}^{660386159}$ & $1.99990_{3651534491}^{4605027210}$ & $2.000095_{765192997}^{974862376}$ & $2.000287_{333375641}^{891951853}$ \\[1mm]
			16 & 46 & $1.9998_{49968886786}^{62414490342}$ & $1.99995_{1587371242}^{2538623688}$ & $2.00004_{7410478972}^{8460188579}$ & $2.00014_{2777979709}^{4839415086}$ \\
		\end{tabular}
	\end{subtable}
\end{table}

\begin{table}[!ht]
	\caption{Maximum radii of the entries of the verified eigenvectors corresponding to the eigenvalues near $2$ for the test problems with \eqref{eq:tridiag}.}
	\begin{subtable}{0.49\textwidth}
		\caption{\texttt{eigs+verifyeig}.}
		\scriptsize
		\centering
		\begin{tabular}{c|llll}		
			$s$ & \multicolumn{1}{c}{$\boldsymbol{x}_1$} & \multicolumn{1}{c}{$\boldsymbol{x}_2$} & \multicolumn{1}{c}{$\boldsymbol{x}_3$} & \multicolumn{1}{c}{$\boldsymbol{x}_4$} \\
			\hline
			5 & \texttt{4.17e-16} & \texttt{5.56e-16} & \texttt{8.61e-16} & \texttt{6.11e-16} \\[1mm]
			6 & \texttt{6.11e-16} & \texttt{8.89e-16} & \texttt{1.28e-15} & \texttt{8.89e-16} \\[1mm]
			7 & \texttt{7.50e-16} & \texttt{1.10e-15} & \texttt{1.67e-15} & \texttt{1.66e-15} \\[1mm]
			8 & \texttt{1.18e-15} & \texttt{1.75e-15} & \texttt{2.56e-15} & \texttt{1.71e-15} \\[1mm]
			9 & \texttt{1.47e-15} & \texttt{2.18e-15} & \texttt{3.31e-15} & \texttt{2.23e-15} \\[1mm]
			10 & \texttt{2.34e-15} & \texttt{3.50e-15} & \texttt{5.10e-15} & \texttt{3.41e-15} \\[1mm]
			11 & \texttt{2.91e-15} & \texttt{4.36e-15} & \texttt{6.62e-15} & \texttt{6.62e-15} \\[1mm]
			12 & \texttt{4.66e-15} & \texttt{6.98e-15} & \texttt{1.02e-14} & \texttt{1.02e-14} \\[1mm]			
			13 & \texttt{5.81e-15} & \texttt{8.71e-15} & \texttt{1.33e-14} & \texttt{1.33e-14} \\[1mm]			
			14 & \texttt{1.40e-14} & \texttt{1.40e-14} & \texttt{2.04e-14} & \texttt{1.36e-14} \\[1mm]
			15 & \texttt{1.75e-14} & \texttt{1.75e-14} & \texttt{2.65e-14} & \texttt{1.77e-14} \\[1mm]
			16 & \texttt{2.80e-14} & \texttt{2.80e-14} & \texttt{4.08e-14} & \texttt{2.72e-14} \\
		\end{tabular}
		\label{tbl:verifyeig_eigvec}
	\end{subtable}
	\begin{subtable}{0.49\textwidth}
		\caption{Hankel matrix approach.}
		\scriptsize
		\centering
		\begin{tabular}{c|llll}
			$s$ & \multicolumn{1}{c}{$\boldsymbol{x}_1$} & \multicolumn{1}{c}{$\boldsymbol{x}_2$} & \multicolumn{1}{c}{$\boldsymbol{x}_3$} & \multicolumn{1}{c}{$\boldsymbol{x}_4$} \\
			\hline
			5 & \texttt{1.17e-11} & \texttt{8.44e-12} & \texttt{1.07e-11} & \texttt{1.40e-11} \\[1mm]
			6 & \texttt{2.81e-10} & \texttt{1.15e-10} & \texttt{1.09e-10} & \texttt{2.48e-10} \\[1mm]
			7 & \texttt{7.56e-10} & \texttt{9.67e-11} & \texttt{1.57e-10} & \texttt{1.21e-09} \\[1mm]
			8 & \texttt{7.79e-11} & \texttt{6.47e-10} & \texttt{1.44e-10} & \texttt{2.27e-10} \\[1mm]
			9 & \texttt{6.37e-10} & \texttt{8.14e-10} & \texttt{2.68e-09} & \texttt{3.95e-10} \\[1mm]
			10 & \texttt{3.89e-08} & \texttt{1.15e-08} & \texttt{2.00e-08} & \texttt{2.39e-08} \\[1mm]
			11 & \texttt{2.60e-08} & \texttt{4.63e-08} & \texttt{3.81e-08} & \texttt{2.33e-08} \\[1mm]
			12 & \texttt{1.44e-06} & \texttt{4.28e-07} & \texttt{5.01e-07} & \texttt{1.01e-06} \\[1mm]
			13 & \texttt{2.44e-06} & \texttt{1.37e-06} & \texttt{1.07e-06} & \texttt{1.96e-06} \\[1mm]
			14 & \texttt{2.00e-06} & \texttt{1.88e-06} & \texttt{3.75e-06} & \texttt{2.94e-06} \\[1mm]
			15 & \texttt{2.66e-05} & \texttt{1.72e-06} & \texttt{1.03e-05} & \texttt{1.71e-05} \\[1mm]
			16 & \texttt{2.44e-05} & \texttt{6.38e-05} & \texttt{3.14e-05} & \texttt{3.17e-05} \\
		\end{tabular}
		\label{tbl:Hankel_eigvec}
	\end{subtable}
	\bigskip
	\begin{subtable}{\textwidth}
		\caption{Rayleigh--Ritz procedure approach.}
		\scriptsize
		\centering
		\begin{tabular}{c|llll}
			$s$ & \multicolumn{1}{c}{$\boldsymbol{x}_1$} & \multicolumn{1}{c}{$\boldsymbol{x}_2$} & \multicolumn{1}{c}{$\boldsymbol{x}_3$} & \multicolumn{1}{c}{$\boldsymbol{x}_4$} \\
			\hline
			5 & \texttt{1.32e-11} & \texttt{9.41e-12} & \texttt{1.19e-11} & \texttt{1.55e-11} \\[1mm] 
			6 & \texttt{3.18e-10} & \texttt{1.30e-10} & \texttt{1.24e-10} & \texttt{2.82e-10} \\[1mm] 
			7 & \texttt{8.53e-10} & \texttt{1.11e-10} & \texttt{1.81e-10} & \texttt{1.38e-09} \\[1mm] 
			8 & \texttt{8.69e-11} & \texttt{7.22e-10} & \texttt{1.62e-10} & \texttt{2.54e-10} \\[1mm] 
			9 & \texttt{6.76e-10} & \texttt{8.24e-10} & \texttt{2.86e-09} & \texttt{3.96e-10} \\[1mm] 
			10 & \texttt{4.07e-08} & \texttt{1.20e-08} & \texttt{2.09e-08} & \texttt{2.50e-08} \\[1mm] 
			11 & \texttt{2.56e-08} & \texttt{4.45e-08} & \texttt{3.63e-08} & \texttt{2.23e-08} \\[1mm] 
			12 & \texttt{1.08e-06} & \texttt{3.23e-07} & \texttt{3.97e-07} & \texttt{8.05e-07} \\[1mm] 
			13 & \texttt{1.84e-06} & \texttt{1.03e-06} & \texttt{7.99e-07} & \texttt{1.48e-06} \\[1mm] 
			14 & \texttt{1.38e-06} & \texttt{1.28e-06} & \texttt{2.56e-06} & \texttt{2.06e-06} \\[1mm] 
			15 & \texttt{1.60e-05} & \texttt{1.16e-06} & \texttt{6.61e-06} & \texttt{9.84e-06} \\[1mm] 
			16 & \texttt{1.35e-05} & \texttt{3.61e-05} & \texttt{1.71e-05} & \texttt{1.76e-05} \\
		\end{tabular}
		\label{tbl:RR_eigvec}
	\end{subtable}
	\label{tbl:tridiag_eigvec}
\end{table}

\subsection{Multiple eigenvalue}
To show the verification performance of the proposed method in the presence of multiple eigenvalues, we test on the problem with matrices
\begin{align}
	A = \mathrm{diag}(0, 0, \dots, 0, 1, 1, 1+\varepsilon, 2, 3, 4) \in \mathbb{R}^{n \times n}, \quad B = \mathrm{I}_n, \quad \varepsilon = 10^{-s}, \quad s = 1, 2, \dots, \quad n = 100,
	\label{eq:multiple}
\end{align}
which has eigenvalues~$0$ with multiplicity $n-6$, $1$ with multiplicity~$2$, $1+\varepsilon$ with multiplicity~$1$ for $\varepsilon \ne 0$, and simple eigenvalues~$2$, $3$, and $4$.
The verified eigenvalues of interest are located in a circle with center~$2.5$ and radius~$2$, i.e., six eigenvalues exist in the circle.
Hence, we set the values of parameters~$\ell = 3$ and $m = 2$, i.e., $\ell m = 6$.
A rigorous bound of the quantity~$| \hat{\lambda} - \gamma |$ required in line~3 of Algorithm~\ref{alg:verifiedSSRR} is computed by using the INTLAB function~\texttt{isregular} hereafter.
The solutions of linear systems~$(z_j B - A) Y_j = BV$, $j = 1, 2, \dots, N$, are rigorously evaluated in line~4 of Algorithm~\ref{alg:verifiedSSRR} by using MATLAB function~\texttt{mldivide} hereafter.

Table~\ref{tbl:eigval_multiple} gives the interval radii of the verified eigenvalues for the test problem with \eqref{eq:multiple} for $s = 1, 2, \dots, 8$.
Table~\ref{tbl:eigvec_multiple} gives the maximum interval radii of the entries of the verified eigenvectors for the test problem with \eqref{eq:multiple} for $s = 1, 2, \dots, 8$.
These tables show that the proposed method works and is robust even in the presence of multiple and nearly multiple eigenvalues.
As the number of $s$ increases, the interval radii of the verified eigenpairs tend 
\begin{table}[ht!]
	\caption{Interval radii of the verified eigenvalues for the test problems with \eqref{eq:multiple} with multiple eigenvalues.}
	\label{tbl:eigval_multiple}
	\scriptsize
	\centering
	\begin{tabular}{c|llllll}		
		& \multicolumn{6}{c}{true eigenvalue} \\
		$s$ & \multicolumn{1}{c}{$1$} & \multicolumn{1}{c}{$1$} & \multicolumn{1}{c}{$1+\varepsilon$} & \multicolumn{1}{c}{$2$} & \multicolumn{1}{c}{$2$} & \multicolumn{1}{c}{$2$} \\
		\hline
		1 & \texttt{1.06e-10} & \texttt{1.06e-10} & \texttt{7.86e-13} & \texttt{2.83e-11} & \texttt{1.30e-11} & \texttt{5.03e-12} \\[1mm]
		2 & \texttt{1.05e-10} & \texttt{1.05e-10} & \texttt{7.33e-13} & \texttt{2.81e-11} & \texttt{1.37e-11} & \texttt{5.03e-12} \\[1mm]
		3 & \texttt{1.05e-10} & \texttt{1.06e-10} & \texttt{7.30e-13} & \texttt{2.82e-11} & \texttt{1.38e-11} & \texttt{5.04e-12} \\[1mm]
		4 & \texttt{1.08e-10} & \texttt{1.08e-10} & \texttt{7.29e-13} & \texttt{2.81e-11} & \texttt{1.38e-11} & \texttt{5.03e-12} \\[1mm]
		5 & \texttt{1.05e-10} & \texttt{1.05e-10} & \texttt{7.29e-13} & \texttt{2.81e-11} & \texttt{1.38e-11} & \texttt{5.02e-12} \\[1mm]
		6 & \texttt{1.08e-10} & \texttt{1.08e-10} & \texttt{7.29e-13} & \texttt{2.81e-11} & \texttt{1.38e-11} & \texttt{5.03e-12} \\[1mm]
		7 & \texttt{1.06e-10} & \texttt{1.06e-10} & \texttt{7.32e-13} & \texttt{2.82e-11} & \texttt{1.38e-11} & \texttt{5.03e-12} \\[1mm]
		8 & \texttt{1.11e-10} & \texttt{1.11e-10} & \texttt{7.69e-13} & \texttt{2.82e-11} & \texttt{1.38e-11} & \texttt{5.03e-12} \\[1mm]
		9 & \texttt{6.67e-10} & \texttt{6.67e-10} & \texttt{6.67e-10} & \texttt{2.80e-11} & \texttt{1.37e-11} & \texttt{5.02e-12} \\[1mm]
		10 & \texttt{1.38e-10} & \texttt{1.38e-10} & \texttt{1.38e-10} & \texttt{2.80e-11} & \texttt{1.37e-11} & \texttt{5.04e-12} \\[1mm]
		11 & \texttt{1.08e-10} & \texttt{1.08e-10} & \texttt{1.08e-10} & \texttt{2.79e-11} & \texttt{1.37e-11} & \texttt{5.01e-12} \\[1mm]
		12 & \texttt{1.09e-10} & \texttt{1.09e-10} & \texttt{1.09e-10} & \texttt{2.80e-11} & \texttt{1.37e-11} & \texttt{5.02e-12} \\[1mm]
		13 & \texttt{1.09e-10} & \texttt{1.09e-10} & \texttt{1.09e-10} & \texttt{2.79e-11} & \texttt{1.37e-11} & \texttt{5.01e-12} \\[1mm]
		14 & \texttt{1.04e-10} & \texttt{1.04e-10} & \texttt{1.04e-10} & \texttt{2.80e-11} & \texttt{1.37e-11} & \texttt{5.02e-12} \\[1mm]
		15 & \texttt{1.06e-10} & \texttt{1.06e-10} & \texttt{1.06e-10} & \texttt{2.79e-11} & \texttt{1.37e-11} & \texttt{5.00e-12} \\[1mm]
		16 & \texttt{1.05e-10} & \texttt{1.05e-10} & \texttt{1.05e-10} & \texttt{2.79e-11} & \texttt{1.37e-11} & \texttt{5.01e-12} \\
	\end{tabular}
	\caption{Maximum radii of the entries of the verified eigenvectors for the test problems with \eqref{eq:multiple} with multiple eigenvalues.}
	\label{tbl:eigvec_multiple}
	\scriptsize
	\centering
	\begin{tabular}{c|llllll}		
		& \multicolumn{6}{c}{true eigenvalue} \\
		$s$ & \multicolumn{1}{c}{$1$} & \multicolumn{1}{c}{$1$} & \multicolumn{1}{c}{$1+\varepsilon$} & \multicolumn{1}{c}{$2$} & \multicolumn{1}{c}{$2$} & \multicolumn{1}{c}{$2$} \\
		\hline
		1 & \texttt{7.51e-11} & \texttt{8.94e-11} & \texttt{9.60e-10} & \texttt{1.68e-10} & \texttt{1.99e-10} & \texttt{6.26e-11} \\[1mm]
		2 & \texttt{1.33e-10} & \texttt{4.03e-11} & \texttt{1.53e-11} & \texttt{1.64e-10} & \texttt{2.02e-10} & \texttt{6.21e-11} \\[1mm]
		3 & \texttt{1.15e-10} & \texttt{2.81e-11} & \texttt{1.07e-11} & \texttt{1.64e-10} & \texttt{2.02e-10} & \texttt{6.19e-11} \\[1mm]
		4 & \texttt{1.08e-10} & \texttt{3.82e-11} & \texttt{1.04e-11} & \texttt{1.64e-10} & \texttt{2.02e-10} & \texttt{6.21e-11} \\[1mm]
		5 & \texttt{8.09e-11} & \texttt{8.13e-11} & \texttt{1.03e-11} & \texttt{1.64e-10} & \texttt{2.02e-10} & \texttt{6.20e-11} \\[1mm]
		6 & \texttt{9.31e-11} & \texttt{6.74e-11} & \texttt{1.04e-11} & \texttt{1.64e-10} & \texttt{2.03e-10} & \texttt{6.22e-11} \\[1mm]
		7 & \texttt{1.12e-10} & \texttt{2.92e-11} & \texttt{1.04e-11} & \texttt{1.64e-10} & \texttt{2.02e-10} & \texttt{6.20e-11} \\[1mm]
		8 & \texttt{1.13e-10} & \texttt{2.64e-11} & \texttt{1.03e-11} & \texttt{1.64e-10} & \texttt{2.01e-10} & \texttt{6.17e-11} \\[1mm]
		9 & \texttt{5.66e-11} & \texttt{9.96e-11} & \texttt{1.03e-11} & \texttt{1.64e-10} & \texttt{2.01e-10} & \texttt{6.19e-11} \\[1mm]
		10 & \texttt{1.14e-10} & \texttt{2.25e-11} & \texttt{1.03e-11} & \texttt{1.64e-10} & \texttt{2.02e-10} & \texttt{6.20e-11} \\[1mm]
		11 & \texttt{5.02e-11} & \texttt{1.04e-10} & \texttt{1.04e-11} & \texttt{1.64e-10} & \texttt{2.02e-10} & \texttt{6.22e-11} \\[1mm]
		12 & \texttt{1.12e-10} & \texttt{2.96e-11} & \texttt{1.03e-11} & \texttt{1.64e-10} & \texttt{2.02e-10} & \texttt{6.20e-11} \\[1mm]
		13 & \texttt{1.12e-10} & \texttt{3.05e-11} & \texttt{1.01e-11} & \texttt{1.64e-10} & \texttt{2.02e-10} & \texttt{6.20e-11} \\[1mm]
		14 & \texttt{1.12e-10} & \texttt{2.91e-11} & \texttt{1.12e-11} & \texttt{1.64e-10} & \texttt{2.02e-10} & \texttt{6.21e-11} \\[1mm]
		15 & \texttt{1.08e-10} & \texttt{5.13e-12} & \texttt{3.74e-11} & \texttt{1.64e-10} & \texttt{2.02e-10} & \texttt{6.19e-11} \\[1mm]
		16 & \texttt{1.12e-10} & \texttt{3.01e-11} & \texttt{7.65e-12} & \texttt{1.64e-10} & \texttt{2.02e-10} & \texttt{6.20e-11} \\ 
	\end{tabular}
\end{table}
to increase.
Even when the number of $s$ is large, the interval radii do not deteriorate.
The proposed method gives verified multiple eigenvalues~$1$ and $2$ whose interval radii are of order up to $10^{-10}$ and the corresponding verified eigenvectors whose entries have maximum interval radii of order up to $10^{-10}$.
Similar trends are observed for the Hankel matrix approach.

\subsection{Effect of the condition number of \texorpdfstring{$B$}{B}}
To show the verification performance of the proposed method for varying the condition number of $B$, we test on the test matrix pencil~$z B - A$ with matrices
\begin{align}
A  = \mathrm{pentadiag} (1, 2, 3, 2, 1) \in \mathbb{R}^{100 \times 100}, \quad B = \mathrm{diag} (1, 1, \dots, 1, b_{100}) \in \mathbb{R}^{100 \times 100}, \label{eq:pentadiag}
\end{align}
where $\mathrm{pentadiag} (\cdot, \cdot, \cdot)$ denotes the pentadiagonal Toeplitz matrix consisting of a pentuple.
To see the effect of the condition number of $B$ on verification performance, the value of an entry~$b_{100}$ varies among $0$, $10^{-16}$, $10^{-15}$, \dots, $10^0$, i.e., the condition number of $B$ associated with the Euclidean norm is $1$, $10^{16}$, $10^{15}$, \dots, $1$, respectively.
There exist exactly six eigenvalues in the interval~$[0.95, 1.05]$ on the real axis and we compute and verify these eigenvalues so that we set the numbers of parameters~$\ell = 3$, $m = 2$ and the interval~$\Omega = [0.95, 1.05]$.
The input arguments of \texttt{eigs} are set to compute the six eigenvalues closest to one on the real axis.

Figure~\ref{fig:rad_eigval} shows the radius of the verified inclusion of each eigenvalue versus the value of $b_{100}$.
We determine the smallest $N$ that satisfies \eqref{eq:N4eigval}.
This figure shows that \texttt{eigs}+\texttt{verifyeig} gives the smallest radius, while the Rayleigh--Ritz procedure approach gives the largest radius.
The interval radii slightly increase for $b_{100} = 10^{-2}$ and $1$.

Figure~\ref{fig:rad_eigvec} shows the maximum interval radius of the entries of the verified eigenvector versus the value of $b_{100}$.
We determine the smallest $N$ that satisfies \eqref{eq:N4eigvec}.
This figure shows that \texttt{eigs}+\texttt{verifyeig} gives the smallest radius, while the Rayleigh--Ritz procedure approach gives the largest radius, similarly to the case of verifying eigenvalues.
The maximum interval radii slightly increase for $b_{100} = 10^{-2}$ and $1$.
These results show that the complex moment approaches work when the matrix~$B$ is ill-conditioned and even semidefinite.
Note that the horizontal axes in the above figures use the logarithmic scale so that the plots for $b_{100} = 0$ are presented for $\log_{10} b_{100} = - \infty$ for convenience.

\subsection{Nearly singular matrix pencils}
To show the verification performance of the proposed method when applied to nearly singular pencils, we test on the problem with
\begin{align}
	A = \mathrm{diag}(0, 1, 2, \dots, n-1) \in \mathbb{R}^{n \times n}, \quad
	B = \varepsilon \oplus \mathrm{I}_{n-1}, \quad n = 100,
	\label{eq:nearlysingular}
\end{align}
which form a nearly singular pencil~$zB-A$ for a small value of $\varepsilon > 0$, as $\det (zB-A) = \varepsilon \prod_{i=0}^{n-1} (z-i)$.
Suppose that the eigenvalues of interest are located in a circle with center~$3$ and radius~$3$.
Hence, there exist exactly six eigenvalue~$1,2, \dots, 6$ in the circle and we compute and verify these eigenvalues so that we set the numbers of parameters~$\ell = 3$ and $m = 2$.

Table~\ref{tbl:eigval_nearlysingular} gives the interval radii of the verified eigenvalues for the test problem with \eqref{eq:nearlysingular} for $\varepsilon = 10^{-s}$, $s = 1, 2, \dots, 16$.
Table~\ref{tbl:eigvec_nearlysingular} gives the maximum interval radii of the entries of the verified eigenvectors for the test problem with \eqref{eq:multiple} for $\varepsilon = 10^{-s}$, $s = 1, 2, \dots, 16$.
These tables show that the proposed method works and is robust even for nearly singular pencils.
Similar trends are observed for the Hankel matrix approach.

\subsection{Practical problem}
To show the verification performance of the proposed method, we test on a practical problem PPE354 obtained from \cite{HoshiImachiKuwataKakudaFujitaMatsui2019JJIAM}, whose size is of $354$ and coefficient matrix~$B$ is not diagonal.
Note that the preceding experiments are performed only on diagonal matrices~$B$.
The problem arises in an organic polymer poly-(phenylene-ethynylene) (PPE) in the para (linear-chain) structure with ten monomers or $120$ atoms.
There exist exactly ten eigenvalues in the interval~$[2.281, 2.428]$, and we compute and verify these eigenvalues so that we set the numbers of parameters~$\ell = 5$ and $m = 2$.
The input arguments of \texttt{eigs} are set to compute the ten eigenvalues closest to $2.3545$ on the real axis.

Table~\ref{tbl:eigval_practical} gives the interval radii of the verified eigenvalues for PPE354.
Table~\ref{tbl:eigvec_practical} gives the maximum interval radii of the entries of the verified eigenvectors for PPE354.
All the methods succeed in the verification.
These figures show that the Rayleigh--Ritz procedure approach gives slightly larger interval radii than the Hankel matrix approach.

\begin{figure}[H]
	\centering
	\begin{minipage}{0.44\hsize}
		\centering
		\includegraphics[width=\textwidth]{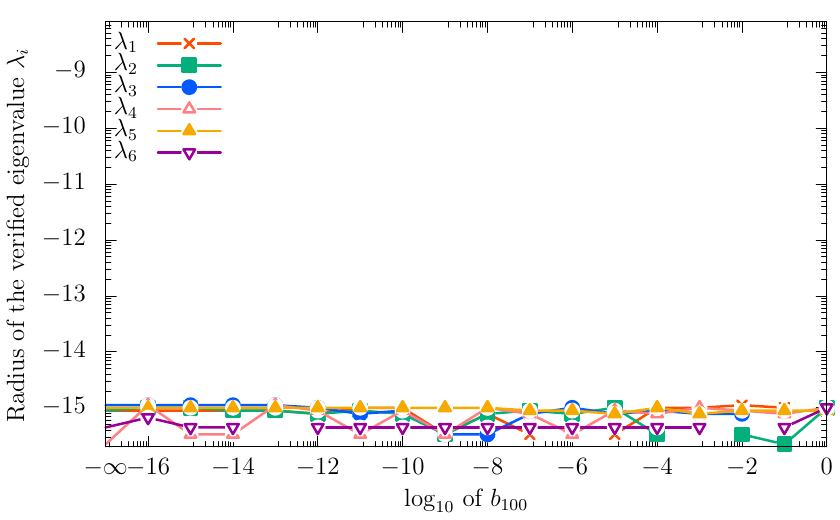}
		\subcaption{\texttt{eigs}+\texttt{verifyeig}.}
		\label{fig:rad_eigs}
	\end{minipage}
	\quad
	\begin{minipage}{0.44\hsize}
		\centering
		\includegraphics[width=\textwidth]{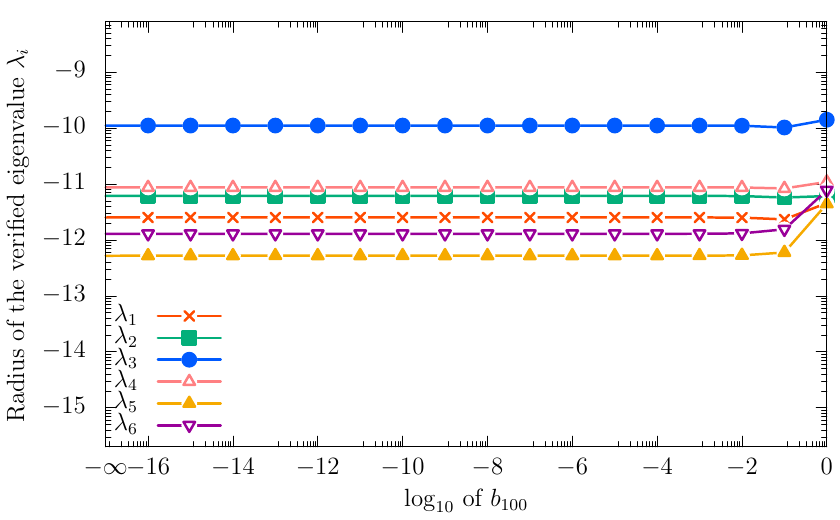}
		\subcaption{Hankel matrix approach.}
		\label{fig:rad_Hankel}
	\end{minipage}
	\begin{minipage}{0.44\hsize}
		\centering
		\includegraphics[width=\textwidth]{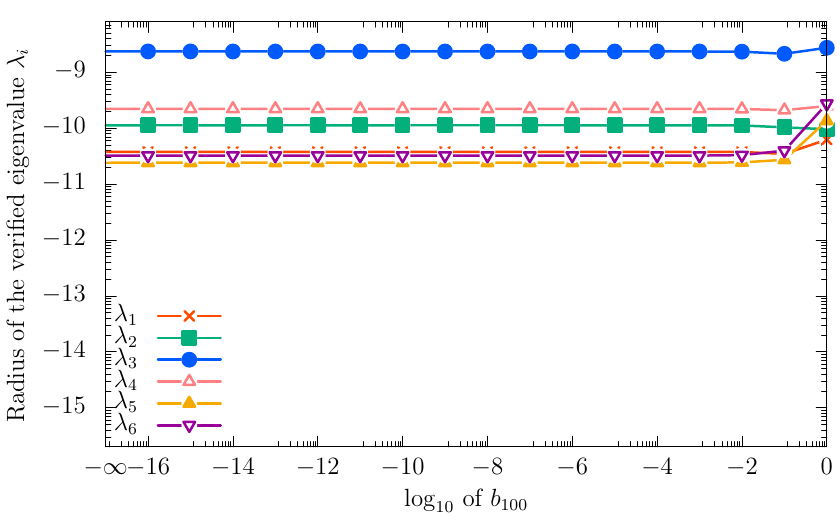}
		\subcaption{Rayleigh--Ritz procedure approach.}
		\label{fig:rad_RR}
	\end{minipage}
	\caption{Radii of the verified eigenvalues for the test problems with \eqref{eq:pentadiag} with ill-conditioned or semidefinite $B$.
		Each symbol represents an eigenvalue with the same index.}
	\label{fig:rad_eigval}
	\centering
	\begin{minipage}{0.44\hsize}
		\centering
		\includegraphics[width=\textwidth]{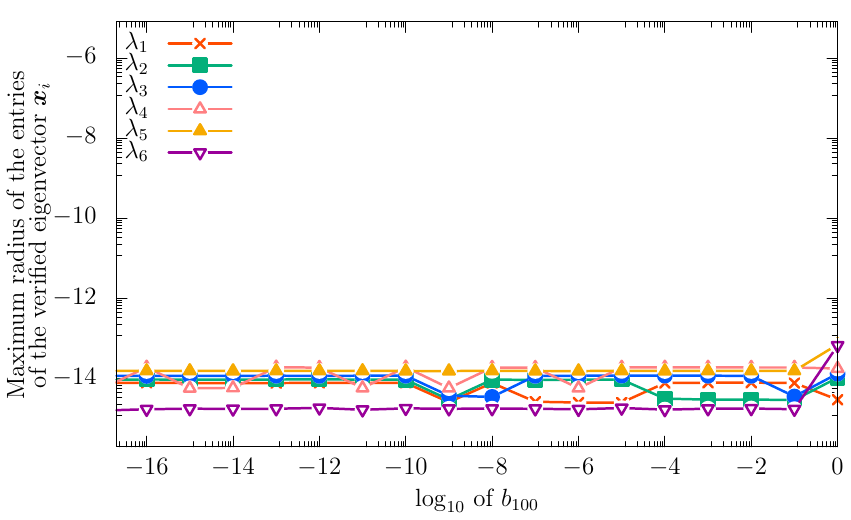}
		\subcaption{\texttt{eigs}+\texttt{verifyeig}.}
		\label{fig:rad_eigvec_eigs}
	\end{minipage}
	\quad
	\begin{minipage}{0.44\hsize}
		\centering
		\includegraphics[width=\textwidth]{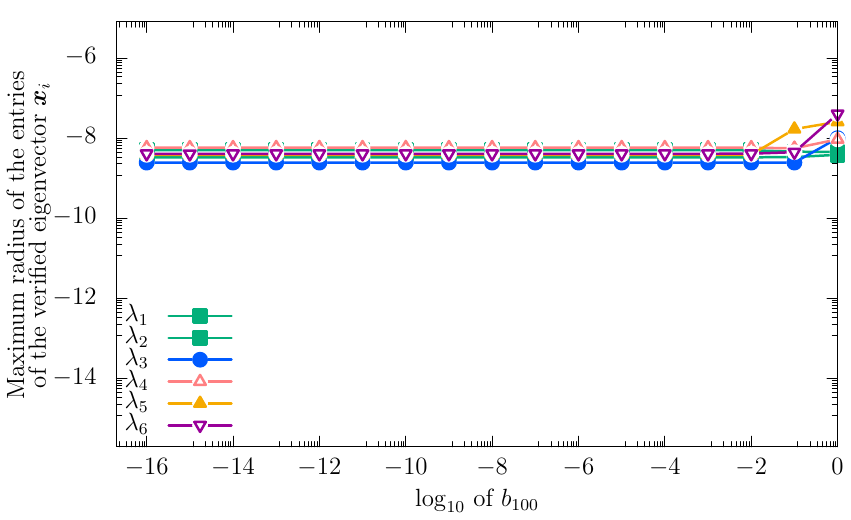}
		\subcaption{Hankel matrix approach.}
		\label{fig:rad_eigvec_Hankel}
	\end{minipage}
	\begin{minipage}{0.44\hsize}
		\centering
		\includegraphics[width=\textwidth]{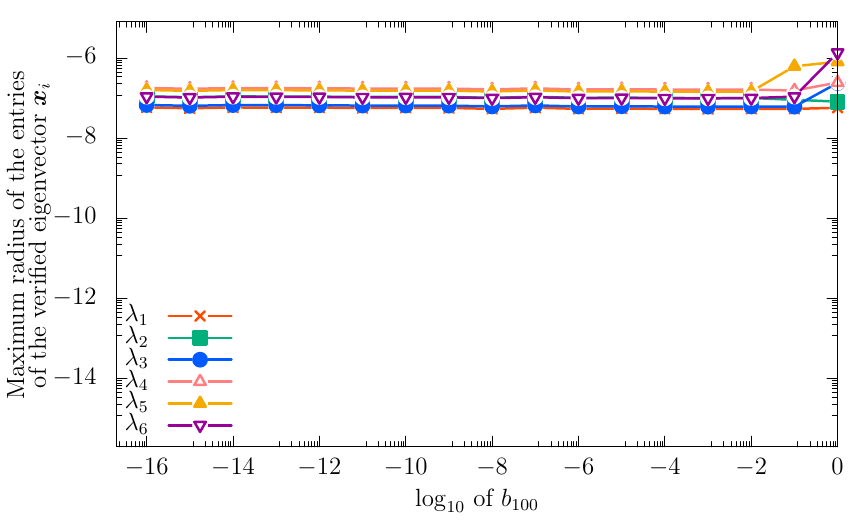}
		\subcaption{Rayleigh--Ritz procedure approach.}
		\label{fig:rad_eigvec_RR}
	\end{minipage}
	\caption{Maximum radius of the entries of the verified eigenvectors for the test problems with \eqref{eq:pentadiag} with ill-conditioned or semidefinite $B$.
		Each symbol represents an eigenvector corresponding to an eigenvalue with the same index.}
	\label{fig:rad_eigvec}
\end{figure}

\begin{table}[H]
	\caption{Interval radii of the verified eigenvalues for the test problems with nearly singular matrix pencils~\eqref{eq:nearlysingular}.}
	\label{tbl:eigval_nearlysingular}
	\scriptsize
	\centering
	\begin{tabular}{c|llllll}		
		& \multicolumn{6}{c}{true eigenvalue} \\
		$s$ & \multicolumn{1}{c}{$1$} & \multicolumn{1}{c}{$2$} & \multicolumn{1}{c}{$3$} & \multicolumn{1}{c}{$4$} & \multicolumn{1}{c}{$5$} & \multicolumn{1}{c}{$6$} \\
		\hline
		1 & \texttt{1.64e-11} & \texttt{2.61e-11} & \texttt{3.89e-11} & \texttt{5.84e-10} & \texttt{1.44e-11} & \texttt{1.27e-10} \\[1mm]
		2 & \texttt{1.64e-11} & \texttt{2.61e-11} & \texttt{3.89e-11} & \texttt{5.84e-10} & \texttt{1.44e-11} & \texttt{1.27e-10} \\[1mm]
		3 & \texttt{1.64e-11} & \texttt{2.61e-11} & \texttt{3.89e-11} & \texttt{5.84e-10} & \texttt{1.44e-11} & \texttt{1.27e-10} \\[1mm]
		4 & \texttt{1.64e-11} & \texttt{2.61e-11} & \texttt{3.89e-11} & \texttt{5.84e-10} & \texttt{1.44e-11} & \texttt{1.27e-10} \\[1mm]
		5 & \texttt{1.64e-11} & \texttt{2.61e-11} & \texttt{3.89e-11} & \texttt{5.84e-10} & \texttt{1.44e-11} & \texttt{1.27e-10} \\[1mm]
		6 & \texttt{1.64e-11} & \texttt{2.61e-11} & \texttt{3.89e-11} & \texttt{5.84e-10} & \texttt{1.44e-11} & \texttt{1.27e-10} \\[1mm]
		7 & \texttt{1.64e-11} & \texttt{2.61e-11} & \texttt{3.89e-11} & \texttt{5.84e-10} & \texttt{1.44e-11} & \texttt{1.27e-10} \\[1mm]
		8 & \texttt{1.64e-11} & \texttt{2.61e-11} & \texttt{3.89e-11} & \texttt{5.84e-10} & \texttt{1.44e-11} & \texttt{1.27e-10} \\[1mm]
		9 & \texttt{1.64e-11} & \texttt{2.61e-11} & \texttt{3.89e-11} & \texttt{5.84e-10} & \texttt{1.44e-11} & \texttt{1.27e-10} \\[1mm]
		10 & \texttt{1.64e-11} & \texttt{2.61e-11} & \texttt{3.89e-11} & \texttt{5.84e-10} & \texttt{1.44e-11} & \texttt{1.27e-10} \\[1mm]
		11 & \texttt{1.64e-11} & \texttt{2.61e-11} & \texttt{3.89e-11} & \texttt{5.84e-10} & \texttt{1.44e-11} & \texttt{1.27e-10} \\[1mm]
		12 & \texttt{1.64e-11} & \texttt{2.61e-11} & \texttt{3.89e-11} & \texttt{5.84e-10} & \texttt{1.44e-11} & \texttt{1.27e-10} \\[1mm]
		13 & \texttt{1.64e-11} & \texttt{2.61e-11} & \texttt{3.89e-11} & \texttt{5.84e-10} & \texttt{1.44e-11} & \texttt{1.27e-10} \\[1mm]
		14 & \texttt{1.64e-11} & \texttt{2.61e-11} & \texttt{3.89e-11} & \texttt{5.84e-10} & \texttt{1.44e-11} & \texttt{1.27e-10} \\[1mm]
		15 & \texttt{1.64e-11} & \texttt{2.61e-11} & \texttt{3.89e-11} & \texttt{5.84e-10} & \texttt{1.44e-11} & \texttt{1.27e-10} \\[1mm]
		16 & \texttt{1.64e-11} & \texttt{2.61e-11} & \texttt{3.89e-11} & \texttt{5.84e-10} & \texttt{1.44e-11} & \texttt{1.27e-10} \\
	\end{tabular}
	
	\caption{Maximum interval radii of the entries of the verified eigenvectors for the test problems with nearly singular matrix pencils~\eqref{eq:nearlysingular}.}
	\label{tbl:eigvec_nearlysingular}
	\scriptsize
	\centering
	\begin{tabular}{c|llllll}		
		$s$ & \multicolumn{1}{c}{$1$} & \multicolumn{1}{c}{$2$} & \multicolumn{1}{c}{$3$} & \multicolumn{1}{c}{$4$} & \multicolumn{1}{c}{$5$} & \multicolumn{1}{c}{$6$} \\
		\hline
		1 & \texttt{6.17e-11} & \texttt{2.33e-10} & \texttt{4.85e-10} & \texttt{2.13e-10} & \texttt{3.70e-10} & \texttt{4.50e-10} \\[1mm]
		2 & \texttt{5.81e-11} & \texttt{2.18e-10} & \texttt{4.54e-10} & \texttt{1.99e-10} & \texttt{3.45e-10} & \texttt{4.18e-10} \\[1mm]
		3 & \texttt{6.37e-11} & \texttt{2.41e-10} & \texttt{5.01e-10} & \texttt{2.20e-10} & \texttt{3.82e-10} & \texttt{4.64e-10} \\[1mm]
		4 & \texttt{6.05e-11} & \texttt{2.28e-10} & \texttt{4.73e-10} & \texttt{2.08e-10} & \texttt{3.60e-10} & \texttt{4.36e-10} \\[1mm]
		5 & \texttt{6.58e-11} & \texttt{2.48e-10} & \texttt{5.17e-10} & \texttt{2.27e-10} & \texttt{3.94e-10} & \texttt{4.78e-10} \\[1mm]
		6 & \texttt{6.68e-11} & \texttt{2.52e-10} & \texttt{5.25e-10} & \texttt{2.31e-10} & \texttt{4.01e-10} & \texttt{4.86e-10} \\[1mm]
		7 & \texttt{6.73e-11} & \texttt{2.54e-10} & \texttt{5.30e-10} & \texttt{2.33e-10} & \texttt{4.05e-10} & \texttt{4.89e-10} \\[1mm]
		8 & \texttt{6.34e-11} & \texttt{2.38e-10} & \texttt{4.96e-10} & \texttt{2.17e-10} & \texttt{3.76e-10} & \texttt{4.56e-10} \\[1mm]
		9 & \texttt{7.00e-11} & \texttt{2.64e-10} & \texttt{5.48e-10} & \texttt{2.41e-10} & \texttt{4.16e-10} & \texttt{5.06e-10} \\[1mm]
		10 & \texttt{6.74e-11} & \texttt{2.53e-10} & \texttt{5.25e-10} & \texttt{2.30e-10} & \texttt{3.98e-10} & \texttt{4.82e-10} \\[1mm]
		11 & \texttt{7.23e-11} & \texttt{2.73e-10} & \texttt{5.65e-10} & \texttt{2.48e-10} & \texttt{4.28e-10} & \texttt{5.21e-10} \\[1mm]
		12 & \texttt{7.32e-11} & \texttt{2.76e-10} & \texttt{5.71e-10} & \texttt{2.50e-10} & \texttt{4.32e-10} & \texttt{5.26e-10} \\[1mm]
		13 & \texttt{7.51e-11} & \texttt{2.82e-10} & \texttt{5.84e-10} & \texttt{2.56e-10} & \texttt{4.42e-10} & \texttt{5.38e-10} \\[1mm]
		14 & \texttt{7.62e-11} & \texttt{2.86e-10} & \texttt{5.91e-10} & \texttt{2.59e-10} & \texttt{4.47e-10} & \texttt{5.44e-10} \\[1mm]
		15 & \texttt{7.72e-11} & \texttt{2.90e-10} & \texttt{5.99e-10} & \texttt{2.62e-10} & \texttt{4.53e-10} & \texttt{5.51e-10} \\[1mm]
		16 & \texttt{7.42e-11} & \texttt{2.78e-10} & \texttt{5.73e-10} & \texttt{2.50e-10} & \texttt{4.32e-10} & \texttt{5.24e-10} \\ 
	\end{tabular}
\end{table}

\begin{table}[H]
	\caption{Interval radii of the verified eigenvalues for the practical problem~PPE354.}
	\label{tbl:eigval_practical}
	\scriptsize
	\centering
	\begin{tabular}{c|llllllllll}
		method & \multicolumn{1}{c}{$\lambda_1$} & \multicolumn{1}{c}{$\lambda_2$} & \multicolumn{1}{c}{$\lambda_3$} & \multicolumn{1}{c}{$\lambda_4$} & \multicolumn{1}{c}{$\lambda_5$} & \multicolumn{1}{c}{$\lambda_6$} & \multicolumn{1}{c}{$\lambda_7$}& \multicolumn{1}{c}{$\lambda_8$}& \multicolumn{1}{c}{$\lambda_9$}& \multicolumn{1}{c}{$\lambda_{10}$} \\
		\hline
		\texttt{eigs+verifyeig}  & \texttt{1.91e-14} & \texttt{1.91e-14} & \texttt{1.91e-14} & \texttt{1.96e-14} & \texttt{1.78e-14} & \texttt{1.91e-14} & \texttt{1.96e-14} & \texttt{2.14e-14} & \texttt{1.96e-14} & \texttt{2.05e-14} \\[1mm]
		Hankel & \texttt{4.96e-10} & \texttt{1.39e-09} & \texttt{4.88e-09} & \texttt{1.21e-09} & \texttt{7.40e-10} & \texttt{1.53e-08} & \texttt{1.25e-09} & \texttt{6.73e-09} & \texttt{9.85e-10} & \texttt{5.30e-10} \\[1mm]
		Rayleigh--Ritz & \texttt{6.66e-09} & \texttt{8.94e-09} & \texttt{3.63e-08} & \texttt{9.02e-09} & \texttt{5.85e-09} & \texttt{1.49e-07} & \texttt{1.50e-08} & \texttt{7.95e-08} & \texttt{7.32e-09} & \texttt{6.15e-09} \\
	\end{tabular}

	\caption{Maximum interval radii of the entries of the verified eigenvectors for the practical problem~PPE354.}
	\label{tbl:eigvec_practical}
	\scriptsize
	\centering
	\begin{tabular}{c|llllllllll}
		method & \multicolumn{1}{c}{$\lambda_1$} & \multicolumn{1}{c}{$\lambda_2$} & \multicolumn{1}{c}{$\lambda_3$} & \multicolumn{1}{c}{$\lambda_4$} & \multicolumn{1}{c}{$\lambda_5$} & \multicolumn{1}{c}{$\lambda_6$} & \multicolumn{1}{c}{$\lambda_7$}& \multicolumn{1}{c}{$\lambda_8$}& \multicolumn{1}{c}{$\lambda_9$}& \multicolumn{1}{c}{$\lambda_{10}$} \\
		\hline
		\texttt{eigs+verifyeig} & \texttt{3.85e-13} & \texttt{6.29e-13} & \texttt{1.13e-13} & \texttt{4.01e-13} & \texttt{5.04e-13} & \texttt{5.59e-13} & \texttt{1.27e-13} & \texttt{2.92e-13} & \texttt{1.21e-13} & \texttt{1.37e-13} \\[1mm] 
		Hankel & \texttt{6.36e-07} & \texttt{3.16e-06} & \texttt{1.14e-06} & \texttt{1.85e-06} & \texttt{2.90e-06} & \texttt{6.11e-07} & \texttt{6.87e-07} & \texttt{2.82e-07} & \texttt{1.94e-07} & \texttt{4.39e-07} \\[1mm] 
		Rayleigh--Ritz & \texttt{1.67e-05} & \texttt{1.35e-04} & \texttt{7.79e-05} & \texttt{1.20e-04} & \texttt{1.17e-04} & \texttt{8.14e-05} & \texttt{4.48e-05} & \texttt{4.38e-05} & \texttt{7.80e-06} & \texttt{1.31e-05} \\ 
	\end{tabular}
\end{table}

\section{Conclusions} \label{sec:conc}
We proposed a verified computation method using the Rayleigh--Ritz procedure and complex moments for eigenvalues in a region and the corresponding eigenvectors of generalized Hermitian eigenvalue problems.
We split the error in the approximated complex moment into the truncation error of the quadrature and rounding errors and evaluate each.
The proposed method uses the Rayleigh--Ritz procedure to project a given eigenvalue problem into a reduced one and can use half the number of quadrature points for our previous Hankel matrix approach to reduce truncation errors to the same order.
Moreover, the transformation matrix for the Rayleigh--Ritz procedure enables verification of the eigenvectors.
Numerical experiments showed that the proposed method is faster than previous methods while maintaining verification performance and works even for nearly singular matrix pencils and in the presence of multiple and nearly multiple eigenvalues.
The Rayleigh--Ritz procedures approach inherits several features from the Hankel matrix approach, such as an efficient technique to evaluate the solutions of linear systems and a parameter tuning technique for the number of quadrature points.
The proposed method will be potentially efficient when implemented in parallel.

%\clearpage

\end{document}